\documentclass[11pt]{amsart}
\usepackage{pstricks}
\usepackage{color}
\usepackage{amssymb}

\usepackage{appendix}
\usepackage{longtable}
\usepackage{setspace}

\usepackage{pstricks}
\usepackage{color}

\usepackage[all,cmtip]{xy}

\usepackage{graphicx}
\usepackage{xypic}

\theoremstyle{plain}
\newtheorem{thm}{Theorem}[section]
\newtheorem{theorem}[thm]{Theorem}

\newtheorem{lemma}[thm]{Lemma}
\newtheorem{corollary}[thm]{Corollary}

\newtheorem{set-up}[thm]{Set-up}
\theoremstyle{definition}
\newtheorem{remark}[thm]{Remark}

\newtheorem{definition}[thm]{Definition}
\newtheorem{claim}[thm]{Claim}

\newtheorem{example}[thm]{Example}

\newtheorem{question}[thm]{Question}

\numberwithin{equation}{section}

\newtheorem{algorithm}[thm]{Algorithm}
\newtheorem{strategy}[thm]{Strategy}

\newcommand{\sA}{{\mathcal A}}

\newcommand{\sC}{{\mathcal C}}

\newcommand{\sF}{{\mathcal F}}

\newcommand{\sH}{{\mathcal H}}

\newcommand{\sL}{{\mathcal L}}
\newcommand{\sM}{{\mathcal M}}

\newcommand{\sO}{{\mathcal O}}
\newcommand{\sP}{{\mathcal P}}

\newcommand{\sR}{{\mathcal R}}
\newcommand{\sS}{{\mathcal S}}
\newcommand{\sT}{{\mathcal T}}
\newcommand{\sU}{{\mathcal U}}

\newcommand{\sX}{{\mathcal X}}

\newcommand{\C}{{\mathbb C}}

\newcommand{\F}{{\mathbb F}}

\newcommand{\Q}{{\mathbb Q}}
\newcommand{\R}{{\mathbb R}}

\newcommand{\Z}{{\mathbb Z}}

\title [Enriques surfaces and Entropy]{Minimum positive entropy of complex Enriques surface automorphisms}

\author{Keiji Oguiso} 
\address{Mathematical Sciences, the University of Tokyo, Meguro Komaba 3-8-1, Tokyo, Japan and Korea Institute for Advanced Study, Hoegiro 87, Seoul,
133-722, Korea}
\email{oguiso@ms.u-tokyo.ac.jp}

\author{Xun Yu}
\address{Center for Applied Mathematics, Tianjin University, 92 Weijin Road, Nankai District,
Tianjin 300072, P. R. China.
}
\email{xunyu@tju.edu.cn}
\thanks{The first named author is supported by JSPS Grant-in-Aid (S) No 25220701, JSPS Grant-in-Aid (S) 15H05738, JSPS Grant-in-Aid (B) 15H03611, and by KIAS Scholar Program. The second named author is supported by NSFC (No. 11701413).}

\begin{document}
\begin{abstract} 
We determine the minimum positive entropy of complex Enriques surface automorphisms. This together with McMullen's work completes the determination of the minimum positive entropy of complex surface automorphisms in each class of the Enriques-Kodaira classification of complex surfaces.
\end{abstract}
\maketitle

\tableofcontents


\section{Introduction}

Throughout this paper, we work over the complex number field $\C$. The aim of this paper is to determine the minimum positive entropy of automorphisms of Enriques surfaces. This together with McMullen's work \cite{Mc07}, \cite{Mc11a}, \cite{Mc16} completes the problem to determine the minimum positive entropy of compact K\"ahler surface automorphisms in each class of the Enriques-Kodaira classification of complex surfaces (see \cite{BHPV04} for basics on complex surfaces, and \cite{Mc02a}, \cite{DS05} for basics on complex dynamics we shall use). 

Let $X$ be a smooth compact K\"ahler surface and $f \in {\rm Aut}\, (X)$ an automorphism of $X$. By the fundamental theorem of Gromov-Yomdin, the {\it entropy} $h(f)$ of $f$ is given by
$$h(f) = \log d_1(f) \,\, \ge 0\,\, .$$
Here $d_1(f)  (\ge 1)$ is the {\it first dynamical degree} of $f$, that is, the spectral radius of $f^*|H^2(X, \C)$ when $f \in {\rm Aut}\, (X)$, which coincides with the spectral radius of $f^*|{\rm NS}\, (X)$ when $X$ is projective (see e.g.  \cite[Corollary 1.4]{ES13}). We call $f$ of positive entropy if $h(f) > 0$, i.e., if $d_1(f) > 1$. 

If $X$ admits an automorphism $f$ of positive entropy, then $X$ is either a rational surface or bimeromorphic to one of the following surfaces: a K3 surface, a complex torus of dimension $2$ or an Enriques surface. This important observation is due to Cantat (\cite[Proposition 1]{Ca99}), which relates complex dynamics with algebraic geometry. In the last three cases, we may and will assume that the surface is minimal. This is because any bimeromorphic selfmap of a minimal surface with non-negative Kodaira dimension is a biregular automorphism (see e.g. \cite{BHPV04}) and the first dynamical degree is a bimeromorphic invariant (\cite[Corollaire 7]{DS05}). As one of the referees pointed out,  in complex dynamics it is meaningful to look at automorphisms modulo bimeromorphic conjugacy.

We call $\tau>1$ a {\it Salem number}, if $\tau$ is conjugate to $1/\tau$ and all other conjugates lie on the unit circle $S^1$. The Salem polynomial of $\tau$ is the monic minimal polynomial $S(x) \in \Z[x]$ of $\tau$. The degree of $S(x)$, which we often call the degree of $\tau$, is an even integer. McMullen 
\cite{Mc02a} observed that $d_1(f)$ is a Salem number if $d_1(f) > 1$, i.e., if the entropy is positive. Furthermore, in \cite[Theorem 1.2]{Mc07}, McMullen also proved the following remarkable fact: if $f$ is of positive entropy, then 
$$d_1(f) \ge \lambda_{10} \approx 1.17628\,\, ,$$
where $\lambda_{10}$ is the Salem number whose Salem polynomial is
$$1 + x-x^3 -x^4 -x^5 -x^6 -x^7 +x^9 +x^{10}\,\, .$$
This is the smallest known Salem number. It is called the Lehmer number $\lambda_{10}$. These two observations give unexpected relations between complex dynamics of surface automorphisms and number theory. Since then, relations between surface automorphisms and Salem numbers, such as realizability of Salem numbers as the first dynamical degree of surface automorphisms and the determination of the minimum Salem number obtained in this way in each class of Enriques-Kodaira classification, and so on, have caught much attention by many authors from various view points. Among many works, McMullen has also shown that there are a rational surface, a non-projective K3 surface and a projective K3 surface, with automorphism $f$ such that $d_1(f) = \lambda_{10}$ (\cite[Corollary 1.3]{Mc07}, \cite[Theorem 1.1]{Mc11a}, \cite[Theorem 1.1]{Mc16}; see also \cite{BK09}, \cite{Ue16} for some other aspects of complex dynamics of rational surface automorphisms). For the case of a complex torus, because of the degree reason ($10 > 6 = b_2(X)$, and also $10 >  4 \ge {\rm rank}\, {\rm NS}\, (X)$ when it is projective), there is {\it a priori} no automorphism such that $d_1(f) = \lambda_{10}$, while the minimum is determined for both projective and non-projective complex torus of dimension $2$. They are the minimum Salem number $\lambda_4$ of degree $4$ and the minimum Salem number $\lambda_6$ of degree $6$ respectively (\cite[Theorem 1.3]{Mc11a}). See also \cite{Re11}, \cite{Re12} for more precise informations and Table \ref{tab:smallest} in Appendix A of our paper for the list of the minimum Salem number $\lambda_{2d}$ in each degree $2d \leq 10$. 

Recall that a complex {\it Enriques surface} $S$ is a smooth compact complex surface whose universal cover, which is of degree $2$, is a K3 surface. All Enriques surfaces are projective and they form a ten dimensional moduli. Any Enriques surface admits a genus one fibration and its Jacobian fibraton is a rational elliptic surface. So, Enriques surfaces are close to both K3 surfaces and rational surfaces. Also $b_2(S) = \rho(S) = 10$ for any Enriques surface. In spite of these facts, it has been shown that there is no Enriques surface automorphism $f$ such that $d_1(f) = \lambda_{10}$ (\cite[Theorem 1.2]{Og10a}). Since then, there are several works toward determination of the minimum Salem number realized as $d_1(f)$ of an Enriques surface automorphism $f$ (\cite{Do17}, \cite{Sh17}, \cite{MOR17}). The current best record is due to Dolgachev \cite{Do17}, which is
$$d_1(f) = \lambda_{{\rm Dol}} \approx 2.08101\,\, ,$$
where $\lambda_{{\rm Dol}}$ is the Salem number whose Salem polynomial is
$$1 - x -2x^2 - x^3 + x^4\,\, .$$

\medskip

Our main result is to show the following

\begin{theorem}\label{thm:main}
Let $\tau_8 \approx 1.58234$ be the Salem number whose Salem polynomial is
$$1 - x^2 -2x^3 - x^4 + x^6\,\, .$$
Then $\tau_8$ is the minimum Salem number which is realized as the first dynamical degree of an Enriques surface automorphism. That is, 
$$d_1(f) \geq \tau_8 \approx 1.58234$$
for any Enriques surface automorphism $f$ of positive entropy, and there are an Enriques surface $S$ and an automorphism $f \in {\rm Aut}\,(S)$ such that $d_1(f) = \tau_8$. 
\end{theorem}

\begin{remark}\label{rmk:main} The Salem number $\tau_8$ in Theorem \ref{thm:main} is the 4th smallest Salem number in degree $6$. (See \cite{Mos} for the list of small Salem numbers of small degrees.)
\end{remark}

There are two issues to prove: (i) realizability of $\tau_8$ and (ii) unrealizability of the Salem numbers $\tau < \tau_8$. Once we establish (i),  it follows from a work of Matsumoto-Ohashi-Rams \cite{MOR17} that $\tau$ in (ii) has to be one of seven Salem numbers $\tau_i$ ($1 \le i \le 7$) listed in Table \ref{tab:8candidates} in the Appendix A. 

As in \cite{Mc16}, our method for both (i) and (ii) is a lattice theoretic one being based on the Torelli theorem for the covering K3 surfaces and the automorphism (lifted to the covering K3), and twists and glues of lattices arising from the Salem polynomials and cyclotomic polynomials. In this approach, among other things, our new results, Theorem \ref{thm:positivity} and Theorem \ref{thm:EnriquesK3}, are particularly important for us. They are also crucial to reduce the problem effectively to a computer algebra problem. We believe that these two theorems have their own interest and will be applicable for other problems. 

Let us explain a bit more about these two theorems. As in the case of K3 surface automorphisms (\cite{Mc16}, \cite{BG18}), one of the essential points in geometric realization of an automorphism from a Hodge isometry of the K3 lattice is the preservation of the ample or K\"ahler cone. In lattice theoretic terms, this is the notion of {\it positivity} introduced by McMullen \cite{Mc16} (see also Definition \ref{def:positive} for the precise definition). In general, it is very hard to check positivity. Theorem \ref{thm:positivity} is a new positivity criterion. Our statement of Theorem \ref{thm:positivity} is given in an equivalent form, so that it can be smoothly applied for both realizability and unrealizability. Our proof of Theorem \ref{thm:positivity} is entirely free from computer algebra. However, our resulting formulation is the one which fits well with computer algebra (see Algorithm \ref{alg:positivity}). Another new issue of realizability and unrealizability by an Enriques automorphism is to descend a candidate K3 surface automorphism to the original Enriques surface, i.e., commutativity of the covering involution and a candidate automorphism of the covering K3 surface. This makes our problem much more complicated than realizability or unrealizability by a K3 surface automorphism. To make this process clear and effective, we introduce a new notion, {\it Enriques quadruple} (Definition \ref{def:triple}). This notion is described purely in terms of lattices and their isometries in which the information of a given Salem number is encoded. Theorem \ref{thm:EnriquesK3} shows that the realizability of a prescribed Salem number $\tau$ as the first dynamical degree of an Enriques automorphism is equivalent to the existence of an Enriques quadruple with the same $\tau$. Our proof of this theorem is again entirely free from computer algebra. However, again, our resulting formulation is the one which fits well with computer algebra. 

We then use computer algebra to check the existence of Enriques quadruple with eight Salem numbers $\tau_{i}$ ($1 \le i \le 8$, see Table \ref{tab:8candidates} in Appendix A). This will be done in Sections \ref{sect:eight} and \ref{sect:nine}. It turns out that $\tau_i$ ($1 \le i \le 7$) are unrealizable (Section \ref{sect:nine}), while $\tau_8$ is realizable (Section \ref{sect:eight}). In this way, we complete the proof of Theorem \ref{thm:main}. All computer algebra programs, which are based on \cite{Mc11b}, and the outputs needed in our proof are available from the second named author's home page \cite{Yu}\footnote{enriques.zip with theorem numbers corresponding to the longer version \cite{OY18} of this paper}.   

We conclude Introduction by posing some open questions closely related to our main result. 

\begin{question}\label{q1}
Let $S$ be an Enriques surface with automorphism of the minimal positive entropy $\log \tau_8$. Can one describe nicely a projective model of (some nice) $S$ or a projective model of its covering K3 surface $\tilde{S}$? 
\end{question}

In our construction, the transcendental lattice $T_{\tilde{S}}$ of the covering K3 surface $\tilde{S}$ is ${\rm I}_{2, 2}(4)$ (Theorem \ref{thm:minimum}). See \cite{MO15} and \cite{Do17} for some explicit projective models of the covering K3 surfaces of Enriques surfaces with automorphisms of positive entropy. 

\begin{question}\label{q3}
How about positive characteristic $p > 0$?
\end{question}

As our method is based on the Torelli theorem for complex K3 surfaces, it cannot be applied to consider this question. See e.g. \cite{ES13}, \cite{EO15}, \cite{Xie15}, \cite{BC16}, \cite{BG18}, \cite{Yu18} for some work related to Salem numbers and surface automorphisms in positive characteristics. However, to our best knowledge, there seems no definite answer for a basic question: {\it if there is a Salem number which can be realized only by an automorphism of an Enriques surface in some special characteristic $p >0$, even for $p \ge 3$.}

\medskip

{\bf Acknowledgement.} We would like to express our thanks to Professors Simon Brandhorst, Igor Dolgachev, H\'el\`ene Esnault, Curtis T. McMullen and Hisanori Ohashi for very valuable discussions and encouragement. Our joint work started after hearing Professor Igor Dolgachev's talk on \cite{Do17} at Korea Institute for Advanced Study (KIAS). We also have had opportunities to collaborate at KIAS for two times later. We would like to express our thanks to KIAS for invitations and financial support during our stay, and the second named author would like to express his thanks to the University of Tokyo for invitation, at which the first version of this work has been completed. We also would like to express our thanks to the referees for their many valuable comments to improve the presentation.


\section{Lattices}\label{sect:two}
\noindent

In this section, we recall some basics on lattices which we will use in our paper (for more details see e.g. \cite{Ni80}, \cite{CS99}). Lemma \ref{lem:L(1/2)} will be frequently used in the sequel. 

A {\it lattice} $(L, (*,**))$ is a finitely generated free $\Z$-module $L$, endowed with a $\Z$-valued symmetric bilinear form $( *, **) =(*, **)_L$. For brevity, we often denote $(x,x)$ by $x^2$. We call $L$ an {\it even} (resp. {\it odd}) lattice if $x^2\in 2\Z$ for any $x\in L$ (resp.  $x^2\notin 2\Z$ for some $x\in L$). Let $(e_1,...,e_n)$ be a $\Z$-basis of $L$. We call $((e_i,e_j))_{1\leq i,j \leq n}$ the {\it Gram matrix} of $L$ with respect to $(e_1,...,e_n)$.  The {\it determinant} ${\rm det}(L)$ of $L$ is defined to be the determinant of any Gram matrix of $L$.  The lattice $L$ is non-degenerate if the symmetric bilinear form on $L$ is non-degenerate (equivalently, ${\rm det}(L)\neq 0$). For a sublattice $L^\prime \subset L$, we say $L^\prime$ is a {\it primitive }sublattice of $L$ if $(L^\prime \otimes \Q) \cap L=L^\prime$. If the signature of $L$, which we denote by ${\rm sig}\, L$, is $(1,n-1)$ and $n>1$, then $L$ is called a {\it hyperbolic lattice}. Here the first entry of the signature is the number of positive squares and the second that of negative ones. For a field $k$, we sometimes denote the $k$-linear space $L \otimes k$ by $L_{k}$. For a sublattice $M\subset L$  (resp. an element $x\in L$), we use $M^{\perp}_L$ (resp. $x^{\perp}_L$) to denote the orthogonal complement of $M$ (resp. $x$) in $L$ (we sometimes omit the subscript $L$ if there is no confusion).

\medskip

 For a nonzero $a\in \Q$, if $a(x,y)_{L}\in \Z$ for any $x,y\in L$, then  the lattice $L(a)$ is defined to be the same $\Z$-module as $L$ with the form given by $$(x,y)_{L(a)}:=a(x,y)_L.$$

\medskip

An element $x\in L$ is called a {\it root} if $x^2=-2$. A lattice is called a {\it root} lattice if it is generated by roots. We use  $A_{k}$ ($k\ge 1$), $D_{l}$ ($l\ge 4$), $E_{m}$ ($m=6,7,8$) to denote the negative definite root lattice whose basis is given by the corresponding Dynkin diagram. We use $U$ (resp. $E_{10}$) to denote the unique even unimodular hyperbolic lattice of rank 2 (resp. rank $10$). Let $r$ and $s$ be positive integers. We denote by ${\rm I}_{r, s}$ (resp. ${\rm II}_{r, s}$ with $r\equiv s\; \text{mod }8$), up to isomorphism, the unique odd (resp. even) unimodular lattice of signature $(r, s)$ (See \cite[Chapter V, Part I]{Se73}).

\medskip

For $f\in {\rm O}(L)$, i.e., for any isometry $f$ of the lattice $L$, we denote the characteristic polynomial ${\rm det}(xI-f)$ by $\chi_f(x)$. For any positive integer $k$, we denote the $k$-th cyclotomic polynomial by $\Phi_k(x)$.

\begin{definition}
Let $G$ be a finite abelian group. A {\it quadratic form} on $G$ is a map $$q:G\longrightarrow \Q/2\Z$$ together with a symmetric bilinear form $$b:G\times G\longrightarrow \Q/\Z$$ such that: 

1) $q(nx)=n^2q(x)$ for all $n\in \Z$ and $x\in G$, and 

2) $q(x+x^\prime)-q(x)-q(x^\prime)\equiv 2b(x,x^\prime) \;{\rm mod}\; 2\Z$ for all $x,x^\prime\in G.$ 

Note that,  a quadratic form $q$ on $G$ is uniquely determined by its restriction to the Sylow subgroups $G_p$ of $G$ (see \cite[Proposition 1.2.2]{Ni80}).

\medskip

The {\it length} of $G$, denoted by $l(G)$, is the minimum number of generators of $G$.

\medskip

Let $L$ be a non-degenerate even lattice. The bilinear form of $L$ determines a canonical embedding $L\hookrightarrow L^*={\rm Hom}(L, \Z)$, and we may view $L^*$ as a subset of $L\otimes \Q$. The quotient group $G(L):=L^*/L$ is finite abelian, and we call $G(L)$ the {\it glue group} of $L$, following \cite{Mc16}. For any $x\in L^*$, we use $\overline{x}$ to denote the image of $x$ in $L^*/L$ under the natural quotient map. The ($\Q$-valued) bilinear form on $L^*$ induced by $(*,**)_L$ gives a bilinear form $b_L$ and a quadratic form $q_L$ on $G(L)$ as follows $$ b_L: G(L)\times G(L)\longrightarrow \Q/\Z, \;\; b_L(\overline{x},\overline{y})\equiv (x,y)\; {\rm mod}\; \Z,$$ and $$q_L: G(L)\longrightarrow \Q/2\Z, \;\; q_L(\overline{x})\equiv (x,x)\; {\rm mod}\; 2\Z .$$ We call $q_L$ the {\it discriminant form}  of $L$. For any prime $p$, we use $q_{L,p}$ to denote the restriction of $q_L$ to the Sylow $p$-subgroup $G(L)_p$. Existence of an even lattice with given discriminant form and signature is characterized by \cite[Theorem 1.10.1]{Ni80}.

\medskip

If the glue group $G(L)$ is a $p$-elementary abelian group for some prime $p$, then we say $L$ is a {\it $p$-elementary} lattice (See \cite{RS89} for classification.)

\end{definition}

The following lemma tells us that if the Sylow $p$-subgroup of the glue group of an even lattice is $p$-elementary of maximal length, then the lattice comes from a ``simpler" even lattice. 

\begin{lemma}\label{lem:L(1/2)}
Let $L$ be a non-degenerate even lattice of rank $n$, and let $p$ be a prime number. Suppose $G(L)_p\cong \F_p^n$ ( if $p=2$, we require $b_L(x,x)=0\in \Q/\Z$ for any $x\in G(L)_2\cong \F_2^n$). Then

1) $\frac{1}{p}L\subset L^*$;

2) $L(\frac{1}{p})$ is a well-defined non-degenerate even lattice.
\end{lemma}

\begin{proof}
1) Since $G(L)_p\cong \F_p^n$, then there exists a subgroup $M\subset L^*$ such that $M/L= G(L)_p$. Since $pm\in L$ for any $m\in M$, $M\subset \frac{1}{p}L\subset L^*$. On the other hand, $( \frac{1}{p}L)/L\cong \F_p^n$. Thus, $M= \frac{1}{p}L$.

 2)  Since $ \frac{1}{p}L\subset L^*$, it follows that $\frac{1}{p}(x,y)\in \Z$ for any $x,y \in L$, and hence $L(\frac{1}{p})$ is a well-defined lattice. Since $L$ is non-degenerate, $L(\frac{1}{p})$ is also non-degenerate. 
 
If $p\ge 3$. For any $x\in L(\frac{1}{p})$, since $(x,x)_{L}=p(x,x)_{L(\frac{1}{p})}$ and $(x,x)_{L}$ is even, it follows that $(x,x)_{L(\frac{1}{p})}$ is also even. Thus, $L(1/p)$ is an even lattice.
 
 If $p=2$ and $b_L(x,x)=0\in \Q/\Z$ for any $x\in G(L)_2\cong \F_2^n$, then for any $y\in L(\frac{1}{2})$, 
$$(y,y)_{L(\frac{1}{2})}=\frac{1}{2}(y,y)_{L}=2(\frac{y}{2},\frac{y}{2})_{L^*}\,\, $$ 
Since $b_L(\frac{\overline{y}}{2},\frac{\overline{y}}{2})=0\in \Q/\Z$,   it follows that $(\frac{y}{2},\frac{y}{2})_{L^*}$ is an integer. Thus, $(y,y)_{L(\frac{1}{2})}=2 (\frac{y}{2},\frac{y}{2})_{L^*}$ is an even integer. Thus, $L(\frac{1}{2})$ is an even lattice. \end{proof}


\section{Twists}\label{ss:twists}
\noindent

In this section, following \cite{Mc16}, we discuss lattice automorphisms canonically associated to irreducible reciprocal polynomials. Theorem \ref{thm:twist} 
below is a generalization of \cite[Theorem 5.2]{Mc16}. This generalization will be used to show unrealizability of $\tau_4$ (which is pseudo-simple but not simple) in Section \ref{sect:nine}.  

Let $P(x)\in \Z[x]$ be a monic irreducible reciprocal polynomial of even degree $d=2m$. A {\it $P(x)$-lattice } is a pair $(L, f)$ of a non-degenerate lattice $L$ and an isometry $f\in {\rm O}(L)$ such that the characteristic polynomial $\chi_{f}(x)$ of $f$ is equal to $P(x)$.

\medskip

Let $(L,f)$ be a $P(x)$-lattice. For any nonzero $a\in \Z[f+f^{-1}]$, the new bilinear form $$(v_1,v_2)_a= (a v_1,v_2)$$ defines the new $P(x)$-lattice $(L(a), f)$, and we call $(L(a), f)$ the {\it twist} of $(L,f)$ by $a$.

\medskip

Let $K$ be the number field $\Q[x]/(P(x))$, and let $R(x)$ be the {\it trace} polynomial of $P(x)$, i.e., $R(x)\in \Z[x]$ is the monic polynomial such that $P(x)=x^mR(x+x^{-1})$. We define $k = \Q[y]/(R(y))$. Then $k$ is a subfield of $K$ such that $[K:k] = 2$ under the natural inclusion $k \subset K$ given by $y = x + x^{-1}$. In particular, the extension $k \subset K$ is Galois under $\iota : x \mapsto x^{-1}$. 

The {\it principal} $P(x)$-lattice $(L_0,f_0)$ is defined by $$L_0=\Z[x]/(P(x))\subset K = \Q[x]/(P(x))$$ with the bilinear form $$(g_1,g_2)_{L_0}=\sum_{i=1}^{d} \frac{g_1(x_i)g_2(x_i^{-1})}{R^{\prime}(x_i+x_i^{-1})}\,\, ,$$ where $(x_i)_1^d$ are the roots of $P(x)$ and $R^{\prime}$ denotes the formal derivative of $R(x)$. The action $f_0 \in {\rm O}(L_0)$ is defined by multiplication by $x$. Then $L_0$ is an even lattice with $|{\rm det}(L_0)|=|P(-1)P(1)|$.

\medskip

 As in  \cite{Mc16}, we say $P(x)$ is {\it simple} if the class number of the number field $K$ is 1, $\sO_{K}=\Z[x]/(P(x))$, and $|P(-1)P(1)|$ is square free.  All small Salem numbers in \cite[Table 1]{Mc16} are simple.  In the way of determining the minimum positive entropy of automorphisms of Enriques surfaces, it turns out that we need to consider Salem numbers which are not simple. We say $P(x)$ is {\it pseudo-simple} if the class number of the number field $K$ is 1, $ \sO_{K}=\Z[x]/(P(x))$, and there exists a $P(x)$-lattice $(L^{\prime},f^{\prime})$ such that $|{\rm det}(L^{\prime})|$ is square free. Thus,  $P(x)$ is pseudo-simple if it is simple.

 \begin{remark}
 The polynomial $x^2+1$ is pseudo-simple but not simple, and the following $(x^2+1)$-lattice 
$$L = (\Z e_1 \oplus \Z e_2, ((e_i,e_j)) = \begin{pmatrix} 
    1&0 \\ 
    0&1\\ 
     \end{pmatrix}),\;\; f = 
    \begin{pmatrix} 
    0&-1 \\ 
    1&0\\ 
     \end{pmatrix}\,\, $$ is odd. This example tells us that $L$ can be  an odd lattice.

 \end{remark}

 \medskip
 
 Our main result of this section is the following

      \begin{theorem}\label{thm:twist}
Let  $P(x)$ be a pseudo-simple monic irreducible reciprocal polynomial. Let  $(L,f)$ be a $P(x)$-lattice such that $|{\rm det}(L)|$ is square free. Then every $P(x)$- lattice is isomorphic to a twist $(L(a), f)$ of $(L,f)$, where $a\in \Z[f+f^{-1}] $.
 \end{theorem}
 
 \begin{proof}
 Since the class number of $K$ is 1 and $\Z[f]$ is the full ring of integers, the inner product on the $P(x)$-lattice $(L,f)$ determines an isomorphism 
$$L \cong b L^* \subset L^*$$ 
for some $b\in \sO_K$ satisfying $$|N_{\Q}^K(b)|= |{\rm det}\, (L)|.$$ By the assumption on $L$, this norm is square-free.
 
 Let $(L^{\prime},f^{\prime})$ be another $P(x)$-lattice. Since $\sO_K\cong \Z[x]/(P(x))$ is a PID, $(L^{\prime},f^{\prime})\cong (L,f)$ as $\sO_K$-modules. Then the inner product of $L^{\prime}$ is of the form 
$$(g_1,g_2)_{L^{\prime}} = (ag_1,g_2)_{L}$$ for some element $a\in k$ (see \cite[Page 276, Remark]{GMc02}). Since $a\in L^*\cong \frac{1}{b}L$, it follows that $a\in b^{-1}\sO_K$.
 
We claim $a\in \sO_k$, the ring of algebraic integers in $k$. In fact, we may write  $a\sO_k=IJ^{-1}$, where $I$ and $J$ are relatively prime ideals in $\sO_k$. We can also write $a=c/d$, where $c$ and $d$ are relatively prime elements of $\sO_K$. Then $d\sO_K=J\sO_K$, and hence $$|N_{\Q}^K(d)|=|N_{\Q}^k(J)^2|.$$  Since $d|b$, $N_{\Q}^k(J)^2$ divides $N_{\Q}^K(b)$ which is square free. Thus $d$ is a unit. Then $a\in k\cap \sO_K=\sO_k$. Note that $\sO_K = \Z[f]$ and therefore 
$$\sO_k = \Z[f]^{\iota} = \Z[f+f^{-1}]$$
by definition of $\iota$. Thus, $a\in \Z[f+f^{-1}]$. \end{proof}

 \begin{remark}\label{rmk:psimple}
      Let $\tau_1<\cdots <\tau_8$ be the eight Salem numbers $\# 1- \# 8$ in \cite[Appendix]{MOR17} (See also Table \ref{tab:8candidates} in Appendix A). Then it turns out that $\tau_i$ is simple (resp. pseudo-simple but not simple) for $i=1,2,3,5,6,7$ (resp. $i=4,8$), as one can verify using computer algebra.
 \end{remark}

 We close this section by recalling the notions of feasible prime and Salem factor from \cite{Mc16} and their relations with isometries of the lattice $E_{10}$ (instead of the K3 lattice ${\rm II}_{3, 19}$). Recall that $E_{10}\cong {\rm II}_{1,9}$ the unique even unimodular hyperbolic lattice of rank $10$ up to isomorphism (see e.g. \cite{Do17} for close relations between $E_{10}$ and Enriques surfaces).
 
Let $\tau$ be a Salem number with Salem polynomial $S(x)$ of degree $2d$. We are interested in the conditions for realizability of $\tau$ by isometries of $E_{10}$, which is important in our study of Enriques quadruples (See Section \ref{sect:seven}).

Let $p\in \Z$ be a prime. We say $p$ is a {\it feasible prime} for $S(x)$ if  
\begin{equation}\label{eq:feasible}
p|N=\prod_{\phi (k)\le 10-2d} {\rm res}(S(x),\Phi_k(x)). 
\end{equation}
The difference between this definition and that in \cite{Mc16} comes from the fact that ${\rm rk}\, (E_{10})=10$, while ${\rm rk}\, ({\rm II}_{3, 19}) =22$.

For any positive integer $n$, we use $D(n)$ to denote the minimum $D\ge 0$ such that $\Z^D$ admits an automorphism of order $n$. It satisfies $D(1)=0$, $D(2)=1$, and $D(n)=D(n/2)$ if $n>2$ is even but $n/2$ is odd. In all other cases, we have $$D(p_1^{m_1}\cdots p_s^{m_s})=\sum \phi (p_i^{m_i})$$ for the prime factorization of $n$.

Let $g$ be an isometry of $E_{10}$ such that a Salem number $\tau$ is an eigenvalue of $g$. Then $\chi_g(x)=S(x)C(x)$ for some product $C(x)$ of cyclotomic polynomials (see for instance \cite[Proposition 3.1]{EOY16}). Let 
$$L:={\rm Ker}(S(g))\subset E_{10}\,\, .$$ 
We call $g|L$ the {\it Salem factor} of $g$.

\begin{theorem} (\cite[Theorem 6.2]{Mc16})\label{thm:GluePeriod}
Let $f:L\longrightarrow L$ be the Salem factor for an isometry of $E_{10}$ such that $\chi_f(x)=S(x)$. Then:

1) The integer $|G(L)|$ is divisible only by the feasible primes for $S(x)$;

2) The order $n$ of the natural map $\bar{f}|G(L)$ induced by $f$ satisfies 
$$D(n)\le 10-{\rm deg}(S(x))\,\, ;$$

3) There exists a product of distinct cyclotomic polynomials $C(x)$ such that 
$$C(\bar{f}|G(L))=0\,\, ,\,\, {\rm deg}(C(x))\le 10-{\rm deg}(S(x)).$$
\end{theorem}


\section{Glue} 
\noindent
In this section, we discuss gluing of lattices and isometries, and controlling glue groups via resultants. We refer to \cite[Section 2]{Mc11a}, \cite[Section 4]{Mc16} for more details. Our main result of this section is Theorem \ref{thm:L2withroots}.

\medskip

Let $L_i$ ($i=1,2$) be non-degenerate lattices. Let $H_i$ be a subgroup of $G(L_i)$. We say a map $\phi: H_1\longrightarrow H_2$ is a {\it gluing map} if 1) $\phi$ is an isomorphism of abelian groups, 2) $b_{L_1}(x,y)=-b_{L_2}(\phi(x),\phi(y))$ for any $x,y\in H_1$. 

For any gluing map $\phi:H_1\longrightarrow H_2$,  we define the lattice $L_1\oplus_{\phi}L_2$ by \begin{equation}\label{eq:glue} L_1\oplus_{\phi}L_2:=\{(x,y)\in L_1^*\oplus L_2^*  | \ \overline{x}\in H_1, \overline{y}\in H_2, \text{ and } \phi(\overline{x})=\overline{y}  \}\subset (L_1\oplus L_2)\otimes \Q.
\end{equation}Clearly $L_1\oplus L_2$ is a sublattice of $L_1\oplus_{\phi}L_2$, and $L_i$ is a primitive sublattice of $L_1\oplus_{\phi}L_2$, $i=1,2$. Moreover, $${\rm det}(L_1){\rm det}(L_2)={\rm det}(L_1\oplus_{\phi}L_2)|H_1|^2.$$

\medskip

For a lattice $L \supset L_1\oplus L_2$ of rank ${\rm rk}(L_1)+{\rm rk}(L_2)$, we say $L$ is a {\it primitive extension} of $L_1$ and $L_2$ if both $L_1$ and $L_2$ are primitive sublattices of $L$.  Any primitive extension of $L_1$ and $L_2$ appears as $L_1\oplus_{\phi} L_2$, in other words, any primitive extension of  $L_1$ and $L_2$ can be obtained by {\it gluing} $L_1$ and $L_2$ via a gluing map. 

\medskip

Any isometry $f: L_1\longrightarrow L_2$ of lattices naturally induces an isomorphism of glue groups: $$\overline{f}: G(L_1)\longrightarrow G(L_2).$$ Moreover, $b_{L_1}(x,y)=b_{L_2}(\overline{f}(x),\overline{f}(y))$ for any $x,y\in G(L_1)$.

\medskip

Let $f_i\in {\rm O}(L_i)$, $i=1,2$. If a gluing map $\phi: H_1\longrightarrow H_2$ satisfies 
$$\overline{f_i}(H_i) = H_i\,\, (i=1, 2)\,\, {\rm and}\,\, \phi\circ \overline{f_1}=\overline{f_2}\circ \phi\,\, ,$$
 then $f_1\oplus f_2$ can be naturally extended to an isometry $f_1\oplus_{\phi} f_2\in {\rm O}(L_1\oplus_{\phi}L_2)$. Conversely, for any primitive extension $L$ of $L_1$ and $L_2$, if an isometry $f\in {\rm O}(L)$ satisfies $f(L_i)=L_i$, $i=1,2$, then $f$ must appear as $f_1\oplus_{\phi}f_2$.

\medskip

The following lemma characterizes the Sylow $p$-subgroup (for certain primes $p$) of the glue group and the discriminant-form of the lattice obtained by gluing two isometries.

\begin{lemma}\label{lem:sylowp}
Let $p$ be a prime number. Let $f\in {\rm O}(L)$ be an isometry of a non-degenerate lattice $L$. Let $L_1$ be a primitive $f$-stable non-degenerate sublattice. 
We set 
$$L_2:=L_1^{\perp}\subset L\,\,, \,\, f_i:=f|L_i\,\, (i=1,2).$$ Suppose that 
$$p \nmid  {\rm res}(\chi_{f_1}(x),\chi_{f_2}(x))\,\, ,$$ 
i.e., the resultant of the two polynomials $\chi_{f_1}(x)$ and $\chi_{f_2}(x)$ is not divided by $p$. Then there exists an isomorphism of abelian groups $$\psi: G(L_1)_p\oplus G(L_2)_p\longrightarrow G(L)_p$$ such that 
$$\bar{f} \circ \psi=\psi \circ (\bar{f_1}\oplus\bar{f_2})$$ 
and $\psi$ is an isomorphism between the quadratic forms $q_{L_1,p}\oplus q_{L_2,p}$ and $q_{L,p}$. 
\end{lemma}

\begin{proof}

By the assumptions, $L$ is a primitive extension of $L_1$ and $L_2$, and $f(L_i)=L_i$, $i=1,2$. Thus, there exists a gluing map $\phi : H_1\longrightarrow H_2$, for some $H_i\subset G(L_i)$, $i=1,2$, such that $L=L_1\oplus_{\phi}L_2$ and $f=f_1\oplus_{\phi}f_2$. Since $p \nmid {\rm res}(\chi_{f_1}(x),\chi_{f_2}(x))$, by \cite[Proposition 4.2]{Mc16}, $p\nmid |H_1|$. Then the map $$\psi: G(L_1)_p\oplus G(L_2)_p\longrightarrow G(L)_p$$ given by $\psi(\bar{x_1},\bar{x_2})=\overline{x_1+x_2}$ is well-defined, where $x_i\in L_i^*$ satisfying $\bar{x_i}\in G(L_i)_p$, $i=1,2$. (In fact, let $v\in L=L_1\oplus_{\phi}L_2$, then, by (\ref{eq:glue}), $|H_1|\cdot v=v_1+v_2$ for some $v_i\in L_i$. Thus, $(x_1,v)=\frac{1}{|H_1|}(x_1,v_1)\in\frac{1}{|H_1|}\Z$. On the other hand, $x_1=\frac{1}{p^m}y_1$ for some $m\ge0$ and $y_1\in L_1$, which implies $(x_1,v)=\frac{1}{p^m}(y_1,v)\in \frac{1}{p^m}\Z$. Then $(x_1,v)\in (\frac{1}{|H_1|}\Z\cap\frac{1}{p^m}\Z)=\Z$ (since $p^m$ and $|H_1|$ are coprime). Thus, $x_1\in L^*$. Similarly, $x_2\in L^*$.) Then $\bar{f} \circ \psi=\psi \circ (\bar{f_1}\oplus\bar{f_2})$, and 
$$q_{L_1,p}(\bar{x_1})+q_{L_2,p}(\bar{x_2})=q_{L,p}(\overline{x_1+x_2})=q_{L,p}(\psi(\bar{x_1},\bar{x_2}))$$ for any $\overline{x_i}\in G(L_i)_p$. Since $p\nmid |H_1|$, it follows that $|G(L)_p|=|G(L_1)_p|\cdot |G(L_2)_p|$ and $\psi$ is an isomorphism. \end{proof}

\medskip

In the process of ruling out Salem numbers in Section \ref{sect:nine}, we often face the following problem: for an isometry $f\in {\rm O}(U\oplus E_{10}(2))$ of finite order with characteristic polynomial $C_1(x)C_2(x)$, where $C_1(x)$ and $C_2(x)$ are coprime polynomials in $\Z[x]$, what can we say about invariants (e.g., glue group, signature) of the sublattice ${\rm Ker}(C_i(f))$, $i=1,2$? Motivated by this, we consider the following

\begin{set-up}\label{set-up}  {\it Let $L_i$ ($i=1,2$) be a non-degenerate lattice of rank $r_i$, and let $f_i\in {\rm O}(L_i)$ ($i=1,2$) be an isometry of finite order $n_i$ such that $n_1$ and $n_2$ are coprime. Suppose $H_i\subset G(L_i)$ is a subgroup satisfying $\bar{f_i}(H_i)=H_i$, and suppose  there is a gluing map $\phi: H_1\longrightarrow H_2$ such that the isometry $f_1\oplus f_2$ extends to $f_1\oplus_{\phi} f_2\in {\rm O}(L_1\oplus_{\phi}L_2)$. }
\end{set-up}

\begin{lemma}\label{lem:Ordercoprime}
Under Set-up \ref{set-up}, $\bar{f_i}|H_i={\rm id}_{H_i}$ ($i=1,2$).
\end{lemma}

\begin{proof}
Since $\phi$ is an isomorphism and $\bar{f_2} \circ \phi=\phi \circ \bar{f_1}$, it follows that ${\rm Ord}(\bar{f_1}|H_1)={\rm Ord}(\bar{f_2}|H_2)$. On the other hand, ${\rm Ord}(\bar{f_i}|H_i)$ divides $n_i$. Since $n_1$ and $n_2$ are coprime, $\bar{f_i}|H_i={\rm id}_{H_i}$.\end{proof}

\begin{lemma}\label{lem:p-elementary}
Under Set-up \ref{set-up}, we suppose $\chi_{f_1}(x)=(\Phi_{p^m}(x))^k$, where $p$ is a prime and $m,k>0$. Then $H_i$, $i=1,2$ is a $p$-elementary abelian group.
\end{lemma}

\begin{proof}
First we consider the case $m=1$. Since ${\rm Ord}(f_1)<\infty$ and $\chi_{f_1}(x)=(\Phi_{p}(x))^k$, it follows that 
$$f_{1}^{p-1}+...+f_1+{\rm id}_{L_1}=0.$$ 
Then 
$$\bar{f_1}^{p-1}| H_1+...+\bar{f_1}|H_1+{\rm id}_{H_1}=0.$$ 
By Lemma \ref{lem:Ordercoprime}, $\bar{f_1}|H_1={\rm id}_{H_1}$. Thus, $p {\rm id}_{H_1}=0$, and $H_1$ ($\cong H_2$ ) is a $p$-elementary abelian group.

If $m>1$, replacing $f_i$ by $f_i^{p^{m-1}}$, one reduces to the case $m=1$. \end{proof}

\begin{lemma}\label{lem:Horderp}
Under Set-up \ref{set-up}, we suppose $n_1\in\{5,7,9\}$. Assume that $\chi_{f_1}(x)=\Phi_{n_1}(x)$. Then $|H_1|=|H_2|=1$ or $p$, where $p$ is the unique prime factor of $n_1$, i.e., $p$ is $5$, $7$, $3$ for $n_1 = 5$, $7$, $9$ respectively.
\end{lemma}

\begin{proof}
We prove the case $n_1=5$ in detail. The other two cases can be proved similarly.

Suppose $n_1=5$. We set $H:=\big\{h\in \frac{1}{5}L_1/L_1|\ \bar{f_1}(h)=h\big\}$. Since $\Phi_5(x)=x^4+x^3+x^2+x+1$ and $\Z[x]/(\Phi_5(x))$ is a PID, there exists $e\in L_1$ such that $\{e,f_1(e),f_1^2(e),f_1^3(e)\}$ is a basis of $L_1$. Then, by an easy computation, one obtain that 
$$H=\Big\langle \frac{\overline{e}}{5}+ \frac{\overline{2f_1(e)}}{5}+\frac{\overline{3f_1^2(e)}}{5}+\frac{\overline{4f_1^3(e)}}{5}\Big\rangle\cong \F_5.$$ 
On the other hand, by Lemmas \ref{lem:Ordercoprime}, \ref{lem:p-elementary} and by our assumption, $H_1$ is isomorphic to a subgroup of $H$. Thus, $|H_2|=|H_1|=1$ or $5$. This completes the proof for the case $n_1=5$. \end{proof}

\medskip

The following theorem will play an important role in ruling out Salem numbers in Section \ref{sect:nine}. We will use this theorem to control the transcendental lattice of the covering K3 surface of an Enriques surface with an automorphism of a given entropy.

\begin{theorem}\label{thm:L2withroots}
Under Set-up \ref{set-up}, we suppose $n_1\in\{5,7,9\}$,  
$$\chi_{f_1}(x)=\Phi_{n_1}(x)\,\, ,\,\,  L_1\oplus_{\phi}L_2\cong U\oplus E_{10}(2)\,\, .$$ 
Let $p$ be the unique prime factor of $n_1$. Then 
$$G(L_1)=\F_2^k\oplus \F_p\,\, ,\,\, G(L_2)=\F_2^{10-k}\oplus \F_p$$ 
for some $k\le r_1$. Moreover, if  $k=r_1$ and the signature of $L_1$ is $(2,r_1-2)$, then  $L_2$ has roots.
\end{theorem}

\begin{proof}

Suppose $n_1=5$. Since $\Phi_5(x)$ is simple, by \cite[Theorem 5.2]{Mc16}, $L_1$ is a twist of the principal $\Phi_5(x)$-lattice $L_0$ whose glue group is of order $|\Phi_5(1)\Phi_5(-1)|=5$, say $L_0(a)$. Thus, ${\rm det}(L_1)$ is divided by 5 as ${\rm det}(L_1)={\rm det}L_0\cdot {\rm det}(a)$. Since $L_1$ and $L_2$ glue to $U\oplus E_{10}(2)$ via $\phi: H_1\longrightarrow H_2$, it follows that 
$$|{\rm det}(L_1){\rm det}(L_2)|=|{\rm det}(U\oplus E_{10}(2))|\cdot |H_1|^2=2^{10}|H_1|^2.$$
Thus, by Lemma \ref{lem:Horderp},  $|H_1|=|H_2|=5$. Then by Lemma \ref{lem:sylowp}, $G(L_1)\cong \F_2^k\oplus \F_5$ for some $k\le r_1=4$, $G(L_2)\cong \F_2^{10-k}\oplus \F_5$.  If $k=4$ and the signature of $L_1$ is $(2,2)$, by Lemma \ref{lem:L(1/2)}, $L_1(1/2)$ is a well-defined even lattice of determinant 5 and signature $(2,2)$. Then  $L_1(1/2)\cong L_0$ since ${\rm sig}(L_0)=(2,2)$ (see \cite{RS89}, \cite{CS99}). Thus, $L_1\cong L_0(2)$. There exists an even lattice $S_1$ with ${\rm sig(S_1)=(0,8)}$ and $G(S_1)\cong \F_2^6\oplus \F_5$. Moreover, the bilinear form $b_{S_1}$ on $G(S_1)$ satisfies two conditions: 

1) $b_{S_1}| G(S_1)_5$ is isomorphic to $-b_{L_0(2)}|G(L_0(2))_5$, and

2) $b_{S_1}(v,v)=0\in \Q/\Z$ for any $v\in G(S_1)_2$. 

Since $U\oplus E_{10}(2)$ is obtained by gluing $L_1$ along $H_1$ (note that $2\nmid |H_1|=5$) to the lattice $L_2$, it follows that $b_{L_2}|G(L_2)$ also satisfies both 1) and 2) (see Lemma \ref{lem:sylowp}). Then $b_{L_2}|G(L_2)$ and $b_{S_1}|G(S_1)$ are isomorphic since their restriction to Sylow subgroups are isomorphic (cf. \cite[Section 3]{Mc11a}). Thus, by \cite[Theorems 1.10.1 and 1.11.3]{Ni80}, $L_2$ is an even lattice of the same signature and the same discriminant form as $S_1$. Hence, by \cite[Corollary 1.9.4]{Ni80},  the genus of $L_2$ is the same as that of $S_1$. 

By Magma (\cite{BCP}), there are, up to isomorphism, exactly two such lattices: $S_1$ and $S_2$ (their Gram matrices can be found in the text file {\it Theorem4.6.txt} which is contained in the ZIP file {\it enriques.zip} available at \cite{Yu}). Both of them have roots (this can be verified by PARI/GP (\cite{Th})), which implies $L_2$ has roots. 

This completes the proof for the case $n_1=5$.
      
Suppose $n_1=7$ or $9$. Then, similar to the case $n_1=5$, we can prove that $G(L_1)=\F_2^k\oplus \F_p$, $G(L_2)=\F_2^{10-k}\oplus \F_p$, for some $k\le r_1=6$. If $k=6$ and the signature of $L_1$ is $(2,4)$, then $L_2$ is a negative definite even lattice of determinant $2^4p$ and rank $6$, where $p=7$ or $3$. Thus, $L_2$ has roots by \cite[Page 3]{Mo44} (and no need of computer algebra in these two cases). \end{proof}


\section{A new positivity criterion}\label{ss:pos}
\noindent
In this section,  we give a new criterion for positivity (Theorem \ref{thm:positivity}). As mentioned in Introduction, Theorem \ref{thm:positivity} and Algorithm \ref{alg:positivity} are crucial in our proof of Theorem \ref{thm:main}. 

\medskip

Let $(L, (*,**))$ be an even hyperbolic lattice of signature $(1,n)$ (note that a hyperbolic lattice in \cite{Mc16} is of signature $(n,1)$). The {\it positive cone} $\mathcal{P}$ of $L$ is defined to be one of the two connected components of  
$$\{x\in L\otimes \R | \ x^2>0\}. $$ 
Let ${\rm O}^+(L)\subset {\rm O}(L)$ be the subgroup consisting of isometries which preserve $\mathcal{P}$. The positive cone $\mathcal{P}$ is cut into {\it chambers} by the set of all {\it root hyperplanes} defined by $$r^{\perp}:=r^{\perp}_{L_{\R}}=\{x\in L\otimes \R |\ (x,r)=0\},$$ where $r\in L$ and $r^2=-2$. Each chamber is a fundamental domain of the Weyl group $W(L)\subset {\rm O}^+(L)$ which is generated by the reflections 
$$s_r: x\mapsto x+(x,r)r$$ 
corresponding to the roots $r\in L$.

\begin{definition}\label{def:positive}
 Let $f\in {\rm O}^+(L)$. We say $f$ is {\it positive} if there exists a chamber $\sM \subset \sP$ such that $f(\sM)=\sM$. 
 \end{definition}
 
 \begin{example}
 For any automorphism of a complex projective K3 surface, the induced isometry of the Picard lattice is positive since it preserves the ample cone of the surface. In this geometric setting, we define the positive cone to be the connected component containing the ample cone. 
 \end{example}
 
 Positivity of isometries of hyperbolic lattices is a subtle condition (see \cite{Mc16}, \cite{BG18}).

 \begin{definition}
 A root $r\in L$ is called an {\em obstructing root} of $f\in {\rm O}^+(L)$ if there is no $\phi\in {\rm Hom}\, (L\otimes \R,\R)$ such that ${\rm Ker}(\phi)$ is negative definite and $\phi(f^i(r))>0$ for all $i\in \Z$.  
 \end{definition}

The following result is a characterization of positivity in terms of obstructing roots.

 \begin{theorem}(\cite[Theorem 2.2]{Mc16})\label{thm:McMullen}
 An isometry $f\in {\rm O}^+(L)$ is positive if and only if $f$ has no obstructing roots.
 \end{theorem}

As in \cite{Mc16}, the following direct consequence of Theorem \ref{thm:McMullen} will be frequently used in Section \ref{sect:nine}.

\begin{corollary}\label{cor:subpositive}
Let $f\in {\rm O}^+(L)$ be positive. Suppose $L^\prime$ is a hyperbolic sublattice of $L$ such that $f(L^\prime)=L^\prime$. Then $f|L^\prime\in {\rm O}^+(L^\prime)$ is positive.  

\end{corollary}

 \medskip
 
Let $f\in {\rm O}^+(L)$ be of spectral radius $\tau >1$. Then the characteristic polynomial of $f$ can be written as $\chi_f(x)=C(x)S(x)$, where $C(x)\in \Z[x]$ is a product of cyclotomic polynomials and $S(x)\in \Z[x]$ is a Salem polynomial (see \cite{Mc02a}, \cite{Og10b}). In particular,  $\tau$ must be a Salem number, and  both $\tau$ and $\tau^{-1}$ are eigenvalues of $f$ with multiplicity one. Let  
$$v, w \in L\otimes \R $$
 be two nonzero elements such that 
$$f(v)=\tau v\,\, ,\,\, f(w)=\tau^{-1}w\,\, .$$  Since $(v,w)\neq 0$ (this holds because
$v^2 = w^2 = 0$, since the eigenvalue is different from $\pm 1$, and otherwise the subspace generated by $v, w$ would be totally isotropic), replacing $v$ by $-v$ if necessary, we may and will assume that $(v,w)>0$. 

Let $h\in L$ such that $h^2>0$. We set 
$$\sC:=\{ r\in L | \ r^2=-2, \text{ and  }r+f(r)+...+f^i(r)=0 \text{ for some } i\ge 1\},$$ 
$$\sR_h:=\{r\in L |\  r^2=-2 \text{ and } (r,h)=0\},$$ 
$$\sS_h:=\{r\in L | \ r^2=-2 \text{ and } (r,h)(r,f(h))<0\}.$$  A root in $\sC$ is called a cyclic root. Cyclic roots are obstructing roots. 

 In general, the hyperbolic lattice $L$ may have infinitely many roots (even infinitely many ``positive'' roots) and checking positivity is a very hard problem. The crucial part in our formulation in Theorem \ref{thm:positivity} is to use $h$ to test for obstructing and to take into account $f$-orbits of roots through the three sets $\sC, \sR_h, \sS_h$, which turn out to be finite (Theorem \ref{thm:positivity}) and to fit very well with computer algebra (Algorithm \ref{alg:positivity}).

\begin{theorem}\label{thm:positivity}
Let $f\in {\rm O}^+(L)$ be of spectral radius $\tau >1$. Let $h\in L$ such that $h^2>0$. Then 

1) The three sets $\sC$, $\sR_h$, $\sS_h$ are finite sets.

2) $f$ is positive if and only if both of the following two conditions are satisfied:

i) $\sC$ is empty,

ii) $(r,v)(r,w)\ge 0$, for all  $r\in \sR_h\cup\sS_h$.

\end{theorem}

\begin{proof}

1) As pointed above, we can write $\chi_f(x)=C(x)S(x)$, where $S(x)$ is the minimal polynomial of the Salem number $\tau$. We can write $C(x)=(x-1)^kC_0(x)$ for some $k\ge 0$ such that $C_0(x)\in \Z[x]$ is not divided by $x-1$. Then $\sC$ consists exactly of roots in ${\rm Ker}(C_0(f))$. Since ${\rm Ker}(S(f))$ is hyperbolic, it follows that ${\rm Ker}(C_0(f))$ is negative definite. Thus,  $\sC$ is finite.

Since $L$ is hyperbolic and $h^2>0$, it follows that the orthogonal complement $h^{\perp}\subset L$ is negative definite. Thus, $\sR_h\subset h^{\perp}$ is finite.

Next we show finiteness of $\sS_h$. Let $$\sA:=\{(a,b)|\  a=(r,h) \text{ and } b=(r,f(h)) \text{ for some } r\in \sS_h\}.$$

\begin{claim}\label{clm:finite}
$\sA$ is a finite set.
\end{claim}

\begin{proof}
Let $r\in \sS_h$. To simplify notation, we let $x=(h,h)$, $y=(h,f(h))$, $a=(h,r)$, and $b=(f(h),r)$. We set $B=
   \begin{pmatrix} 
    x&y&a\\ 
    y&x&b\\ 
    a&b&-2\\ 
    \end{pmatrix}$. Then the determinant of $B$ is 
$$-2x^2+2y^2+2aby-xa^2-xb^2\,\, .$$
 Since $L$ is hyperbolic and the three elements $h, f(h), r$ generate a sublattice of $L$, it follows that this determinant is greater than or equal to $0$. Thus, 
    \begin{equation}\label{eq:det}
    -2x^2+2y^2+2aby\ge x(a^2+b^2).
    \end{equation}
    Note that $x>0$, $y>0$ (since $f\in {\rm O}^+(L)$), $ab<0$ (since $r\in \sS_h$). Then $2aby<0$, and the inequality (\ref{eq:det}) implies that both $a$ and $b$ are bounded (for fixed $x$ and $y$). Thus, $\sA$ is finite. This completes the proof of the claim.\end{proof}

For any $(a,b)\in \sA$, we set $$\sS_h^{(a,b)}:=\{r\in \sS_h |\ (r,h)=a, (r,f(h))=b\}.$$ Then $\sS_h^{(a,b)}\subset (-bh+af(h))^{\perp}$. Since $(-bh+af(h))^2>0$ by $ab < 0$, it follows that $(-bh+af(h))^{\perp}$ is negative definite. Thus $\sS_h^{(a,b)}$ is finite. Then $\sS_h=\cup_{(a,b)\in \sA}\sS_h^{(a,b)}$ is also finite.

2) Suppose $f$ is positive. Then $f$ cannot have cyclic roots, i.e., $\sC$ is empty. Let $\sM$ be an $f$-invariant chamber. Since $(v,w)>0$, we may assume both of $v$ and $w$ are contained in the closure of $\sM$ by the Birkhoff-Perron-Frobenius theorem \cite{Bi67}. Then $(r,v)(r,w)\ge 0$ for any root $r\in L$. Thus, the condition ii) is true.

Suppose both of the two conditions i) and ii) are satisfied. Let $$\sT_h=\{ r\in L |\ r^2=-2, \text { and } r\neq f^k(r^\prime) \text{ for any } k\in \Z, r^\prime \in \sR_h\cup\sS_h\},$$ i.e., $\sT_h$ consists of the roots which does not belong to any $f$-orbit of roots in $\sR_h\cup\sS_h$.  

Let $r\in \sT_h$. Then $f^k(r)$ is not in $\sR_h\cup\sS_h$ for any $k\in \Z$. Thus, $(f^k(r),h)(f^{k-1}(r),h)>0$ for any $k$. Then either $(h,f^k(r))>0$ for all $k$, or $(-h,f^k(r))>0$ for all $k$. Thus, $r$ is not an obstructing root. 

Let $r\in \sR_h\cup\sS_h$. There are two possibilities: a) at least one of $(r,v)$ and $(r,w)$ is nonzero, b)  both $(r,v)$ and $(r,w)$ are zero. In case a), by condition ii), interchanging $v$ and $w$ if necessary, we may assume $(r,v)>0$ and $(r,w)\ge 0$, then $(v+w,f^k(r))>0$ for all $k\in \Z$. Since $(v+w,v+w)>0$, it follows that $(v+w)^{\perp}\subset L\otimes \R$ is negative definite. Thus, $r$ is not an obstructing root. In case b), $(v+w, f^k(r))=0$ for all $k\in \Z$. Since $(v+w)^\perp$ is negative definite, it follows that the $f$-orbit of $r$ is a finite set. Then there exists $m>0$ such that $f^m(r)=r$. Let  
$$\alpha=r+f(r)+...+f^{m-1}(r)\,\, .$$ Then $f(\alpha)=\alpha$, $\alpha\neq 0$ (since, by condition i), $r$ is not a cyclic root), and $(\alpha,\alpha)<0$. Since $(f^k(r), \alpha)=(r,\alpha)$ for any $k$, it follows that $(r,\alpha)=\frac{(\alpha,\alpha)}{m}<0$. By $(v+w,\alpha)=0$,  it follows that 
$$(N(v+w)-\alpha,N(v+w)-\alpha)>0$$ for sufficiently large $N>0$. Then $r$ is not an obstructing root since $(N(v+w)-\alpha, f^k(r))>0$ for all $k\in \Z$.  

Note that a root is an obstructing root if and only if some member of its $f$-orbit is an obstructing root. Therefore, we have proved that  if the two conditions i) and ii) are satisfied, then $f$ has no obstructing roots and, by Theorem \ref{thm:McMullen},  $f$ is positive. This completes the proof of the theorem. \end{proof}

For a root $r\in L$ violating ii) in Theorem \ref{thm:positivity}  (i.e., $(r,v)(r,w)<0$), the root hyperplane $r^\perp$ crosses the line segment connecting $v$ and $w$, and $r$ is an obstructing root (see the top and middle of \cite[Fig. 1]{Mc16} for concrete examples of such obstructing roots). 

The following algorithm to check positivity is based on Theorem \ref{thm:positivity}.

\begin{algorithm}\label{alg:positivity}
Input: a pair $(L,f)$ of  an even hyperbolic lattice $L$ and an isometry $f\in {\rm O}^{+}(L)$ of spectral radius $\tau>1$. To check positivity of $f$, proceed as follows.

1. Determine the factorization $\chi_f(x)=(x-1)^kC_0(x)S(x)$, where $k\ge 0$, $(x-1) \nmid C_0(x)$. If $C_0(x)\neq 1$, go to step 2; otherwise, go to step 3.

2. Compute $\sC$ consisting of roots in ${\rm Ker}(C_0(f))\subset L$. If $\sC\neq \emptyset$, then output one $r_0\in\sC$ (thus $f$ is not positive) and stop; otherwise, go to step 3.

3. Compute eigenvectors $v,w$ of $f$ corresponding to $\tau,\tau^{-1}$ respectively. If $(v, w)<0$ then replace $v$ by $-v$.

4. Find one $h\in L$ with $h^2>0$ by taking the integer part of $n(v+w)+z$ for some large enough $n$ and randomly chosen small $z\in L$. (This step may be skipped if such $h$ is already given.)

5. Compute the finite set $\sR_h$ consisting of roots in $ h^\perp_L$.

6. Check if $(r,v)(r,w)\ge0$ for all $r\in \sR_h$ one by one. If $(r_1,v)(r_1,w)<0$ for some $r_1\in \sR_h$, then output $r_1$ (thus $f$ is not positive) and stop. 

7. Find the finite set $$\sA^\prime:=\{(a,b)\in \Z\times\Z |\, -2x^2+2y^2+2aby\ge x(a^2+b^2), a>0, b<0\}$$ where $x=(h,h)>0$, $y=(h,f(h))>0$.

8. Form the finite set $\sH:=\{-bh+af(h) | (a,b)\in \sA^{\prime}\}\subset L$.

9. Run steps 5 and 6  for all $h^\prime\in\sH$ one by one. If some root $r_2\in\sR_{h_1^\prime}$ satisfies $(r_2,v)(r_2,w)<0$ for some $h_1^\prime\in\sH$, output $r_2$ (thus $f$ is not positive) and stop. If otherwise, then output that $f$ is positive and stop. \end{algorithm}

\begin{remark}
1) Note that \begin{equation}\label{eq:Sh}
\sS_h=(\bigcup_{h^\prime\in \sH}\sR_{h^\prime})\setminus \sR_h.
\end{equation} Thus, if no root $r^\prime$ satisfying $(r^\prime,v)(r^\prime,w)<0$ is found in step 9, then both i) and ii) in Theorem \ref{thm:positivity} are satisfied and hence $f$ is positive.

2) If an $h\in L$ with $h^2>0$ is given,  by (\ref{eq:Sh}),  one can find  $\sR_h$ and $\sS_h$ using only three of the steps above: 5 (for $h$), 7, 8, 5 (for all $h^\prime \in \sH$). Clearly, if both $\sR_h$ and $\sS_h$ are empty, then the chamber containing $h$ is $f$-stable and $f$ is positive.

3) For practical purposes, Theorem \ref{thm:positivity} is easy to apply. In fact, all of the nine steps in Algorithm \ref{alg:positivity} often can be easily handled by computer algebra system (we use PARI/GP (\cite{Th}) to find roots in even negative definite lattices). The crucial point is the following: the elements of $\sA^\prime$ and $\sA$ (see Claim \ref{clm:finite}) can be easily found out by computer algebra (note that, for any $(a,b)\in \sA$, either $(a,b)\in \sA^\prime$ or $(-a,-b)\in \sA^\prime$).

\end{remark}


\section{Enriques surfaces and K3 surfaces}
\noindent
In this section, based on close relation between Enriques surfaces and K3 surfaces, we establish two constraints for automorphisms of Enriques surfaces (Lemmas \ref{lem:noroots} and \ref{lem:gL-}). 
\medskip

Let $Y$ be an Enriques surface and  let $X$ be the universal cover of $Y$. Then there exists a fixed point free involution  $\sigma: X\longrightarrow X$ such that $X/\sigma=Y$. Let $\pi: X\longrightarrow Y$ denote the natural quotient map. To simplify notation,  we use $L$ to denote $H^2(X,\Z)$. The isometry $\sigma^*\in {\rm O}(L)$ induced by $\sigma$ is of order 2, and we set $$L^+:=\{\alpha\in L | \ \sigma^*(\alpha)=\alpha \}\,\, ,\,\,L^-:=\{\alpha\in L | \ \sigma^*(\alpha)=-\alpha \}\,\, .$$ Then the lattice $L$ is a primitive extension of $L^+$ and $L^-$, and 
\begin{equation}\label{eq:L+}
L^+\cong E_{10}(2)\,\, ,\,\, L^-\cong U\oplus E_{10}(2)\,\, ,
\end{equation}
cf. \cite{BP83}.
Let $H^2(Y,\Z)_f$ denote the free part of $H^2(Y,\Z)\cong \Z^{10}\oplus \Z/2\Z$. Then,  $H^2(Y,\Z)_f\cong U\oplus E_{8}$, and $\pi^*(H^2(Y,\Z)_f)=L^+$.

\begin{lemma}\label{lem:noroots}
Let $x\in {\rm NS}(X)$. If $x\in (L^+)^{\perp}$, then $(x,x)\neq -2$.
\end{lemma}

\begin{proof}
Since $Y$ is projective,  $H^2(Y, \Z)_f$ contains an ample class, say $h$. Then $\pi^{\ast}(h)\in L^+$ is also ample since $\pi$ is a finite map.

If $x\in {\rm NS}(X)$ and $x^2=-2$, then by Riemann-Roch Theorem, either $x$ or $-x$ is effective. Then either $(x,\pi^{\ast}h)>0$ or $(-x,\pi^{\ast}h)>0$. Thus, $(x,\pi^{\ast}h)\neq 0$, a contradiction to $x\in (L^+)^\perp$. \end{proof}

\medskip

Any automorphism $g\in {\rm Aut}(Y)$ lifts (in two ways) to an automorphism $\hat{g}\in {\rm Aut}(X)$ commuting with $\sigma$. Thus, if we set $${\rm Aut}(X,\sigma):=\{f\in {\rm Aut}(X)|\ f \circ \sigma =\sigma \circ f\},$$ then ${\rm Aut}(Y)={\rm Aut}(X,\sigma)/\{{\rm id}, \sigma\}$.  Since $\hat{g}^*\sigma^*=\sigma^*\hat{g}^*$, both $L^+$ and $L^-$ are $\hat{g}^{\ast}$-stable. We want to understand the relation between  the characteristic polynomials of $\hat{g}^{\ast}| L^+$ and $\hat{g}^{\ast}| L^-$.  

\begin{lemma}(\cite[Lemma 2.2]{MOR17})\label{lem:chiL-}
Let $f\in {\rm O}(U\oplus E_{10}(2))$. Then $$\chi_f(x)\equiv (1+x)^2\chi_{\overline{f}}(x) \;{\rm mod}\; 2.$$
\end{lemma}

The relationship between  $\chi_{{\hat{g}}^\ast |{L^-}}(x)$ and $\chi_{{\hat{g}}^\ast |{L^+}}(x)$ in the following lemma is important for us since it reduces the number of isometries of $L^{\pm}$ which we need to consider to determine whether a given Salem number can be realized by automorphisms of Enriques surfaces.

\begin{lemma}\label{lem:gL-}
The isometry $\hat{g}^\ast |{L^-} \in {\rm O}(L^-)$ is of finite order, and $$\chi_{{\hat{g}}^\ast |{L^-}}(x)\equiv (1+x)^2\chi_{{\hat{g}}^\ast |{L^+}}(x)\equiv (1+x)^2\chi_{g^\ast}(x) \;{\rm mod}\; 2.$$
\end{lemma}

\begin{proof}
By \cite[Proposition (3.2)]{BP83}, $\hat{g}^\ast|_{L^-}$ is of finite order. Since $L^-$ and $L^+$ are orthognal to each other in the unimodular lattice $L$, it follows that $$\chi_{\overline{\hat{g}^\ast|{L^-}}}(x)=\chi_{\overline{\hat{g}^\ast|{L^+}}}(x).$$ By Lemma \ref{lem:chiL-}, $$\chi_{\hat{g}^\ast|{L^-}}(x)\equiv (1+x)^2\chi_{\overline{\hat{g}^\ast|{L^-}}}(x)\;{\rm mod}\; 2$$  Note that $$\chi_{\overline{\hat{g}^\ast|{L^+}}}(x)\equiv \chi_{\hat{g}^\ast|{L^+}}(x)\equiv \chi_{g^\ast}(x)\;{\rm mod}\; 2.$$ Thus,  $$\chi_{{\hat{g}}^\ast |{L^-}}(x)\equiv (1+x)^2\chi_{{\hat{g}}^\ast |{L^+}}(x)\equiv (1+x)^2\chi_{g^\ast}(x)\;{\rm mod}\; 2.$$ This completes the proof of the Lemma. \end{proof}

\medskip


\section{Enriques quadruple and realization conditions}\label{sect:seven}

In this section, we introduce the notion of Enriques quadruple (Definition \ref{def:triple}), and reduce realization problem to purely lattice theoretical problem in term of this notion (Theorem \ref{thm:EnriquesK3}). This reduction is crucial in our proof of the main theorem.

\begin{definition}\label{def:triple} 

Let $L^+$ and $L^-$ be two lattices isometric to $E_{10}(2)$ and $U\oplus E_{10}(2)$ respectively. Let $f^+\in {\rm O}^+(L^+)$, $f^-\in {\rm O}(L^-)$, let $T\subset L^-$ be a primitive sublattice, and let $\phi: G(L^-)\longrightarrow G(L^+)$ be a gluing map.
 
We say the 4-tuple $(f^+,f^-,T, \phi)$ is an {\it Enriques quadruple} if all of  the following eight conditions are satisfied:

\medskip
1) the spectral radius of $f^+$ is a Salem number $\tau$,

2) $\chi_{f^-}(x)\equiv (1+x)^2\chi_{f^+}(x)\; \text{mod } 2$, 

3) $f^-$ is of finite order, 

4) the signature of $T$ is $(2,r)$, where $r\ge 0$,

5) $f^-(T)=T$ and the minimal polynomial of $f^-| T$ is irreducible,
 
6)   $T^{\perp}_{L^-}$ has no roots,

7)  $L^-\oplus_{\phi} L^+\cong {\rm II}_{3,19}$ and $f^-\oplus f^+$ extends to $f^-\oplus_{\phi} f^+\in {\rm O}(L^-\oplus_{\phi} L^+)$,

8) there exists $h\in L^+$ such that: 

i) $(h,h)>0$, 

ii) $h^{\perp}_{T^\perp_{L^-\oplus_{\phi} L^+}}$ has no roots, and

iii) $h$ and $(f^-\oplus_{\phi} f^+)(h)$ are in the same chamber of $T^\perp_{L^-\oplus_{\phi} L^+}$.

\medskip

The {\it entropy} of an Enriques quadruple is defined to be the entropy of $f^+$, i.e., $\log \tau$.
\end{definition}

\begin{remark}
Condition 2) follows from condition 7) (cf. Lemma \ref{lem:gL-}), and clearly condition 6) follows from condition 8). However, we include conditions 2) and 6) in Definition \ref{def:triple}, as we will frequently use them in Section 9.

\end{remark}

\medskip

The next lemma is known for the experts, and is crucial in our proof of Theorem \ref{thm:EnriquesK3}. We give a proof in the Appendix B.

\begin{lemma}\label{lem:P}
Let $T$ be a lattice of signature $(2,r)$, where $0\le r\le 10$. Let $f\in {\rm O}(T)$ be an isometry of finite order such that the minimal polynomial  of $f$ is irreducible. Then $T_{\R}$ contains an $f$-invariant plane $P$ such that $P$ has signature $(2,0)$, $f_{\R}|P\in {\rm SO}(P)$, and $P^{\perp}_{T_{\R}}\cap T=0$.
\end{lemma}

The main result of this section is the following:

\begin{theorem}\label{thm:EnriquesK3}
A Salem number $\tau$ can be realized by an automorphism of an Enriques surface if and only if there exists an Enriques quadruple of entropy $\log\tau$. 
 
 \end{theorem}

\begin{proof}
Suppose $\tau$ can be realized by an automorphism $g:Y\longrightarrow Y$ of an Enriques surface $Y$. Let $\sigma :X\longrightarrow X$ be the fixed point free involution of the covering K3 surface $X$ such that $X/\sigma=Y$. Let $\pi: X\longrightarrow Y$ be the natural quotient map. Let $T_X$  and $\omega_X\in T_X\otimes \C$ denote the transcendental lattice and a nonzero holomorphic two form on $X$ respectively. Let $\hat{g}\in {\rm Aut}(X)$ denote a lift of $g$. Recall  $$H^2(X,\Z)^{\sigma^*}\cong E_{10}(2),\,\,(H^2(X,\Z)^{\sigma^*})^\perp\cong U\oplus E_{10}(2)$$ see (\ref{eq:L+}). Note that the even unimodular lattice $H^2(X,\Z)$ is a primitive extension of $H^2(X,\Z)^{\sigma^*}$ and $(H^2(X,\Z)^{\sigma^*})^\perp$, and both $H^2(X,\Z)^{\sigma^*}$ and $(H^2(X,\Z)^{\sigma^*})^\perp$ are $\hat{g}^*$-stable. Thus,  there exists a gluing map $$\phi: G((H^2(X,\Z)^{\sigma^*})^\perp)\longrightarrow G(H^2(X,\Z)^{\sigma^*})$$ such that $$(H^2(X,\Z)^{\sigma^*})^\perp\oplus_{\phi}H^2(X,\Z)^{\sigma^*}=H^2(X,\Z)\cong {\rm II}_{3,19}$$ and  $\hat{g}^*=\hat{g}^*|(H^2(X,\Z)^{\sigma^*})^\perp\oplus_{\phi} \hat{g}^*| H^2(X,\Z)^{\sigma^*}$. To simplify notations, we set $$f^+:=\hat{g}^*| H^2(X,\Z)^{\sigma^*}, \,\, f^-:=\hat{g}^*|(H^2(X,\Z)^{\sigma^*})^\perp.$$ Since the entropy of $g$ is $\log \tau$, the entropy of $f^+$ is also $\log\tau$.

By Lemma \ref{lem:gL-}, $f^-$ is of finite order and $$\chi_{f^-}(x)\equiv (1+x)^2\chi_{f^+}(x)\; \text{mod } 2.$$ Since $T_X$ is $\hat{g}^*$-stable, $T_X$ is also $f^-$-stable. Let $\omega_X$ be a nonzero holomorphic 2-form on $X$. Since $T_X$ is the unique minimal sublattice of $H^2(X,\Z)$ such that $$\C[\omega_X]\in T_X\otimes \C ,$$ the minimal polynomial of $f^-|T_X$ is irreducible. By Lemma \ref{lem:noroots}, the orthogonal complement $N$ to $T_X$ in $(H^2(X,\Z)^{\sigma^*})^\perp$ has no roots. Note that $f^-\oplus_{\phi}f^+=\hat{g}^*$ preserves the ample cone, and $H^2(X,\Z)^{\sigma^*}$ contains an ample class, say $h$. Thus, $h$ and  $(f^-\oplus_{\phi}f^+)(h)$ are in the same chamber of ${\rm NS}(X)=(T_X)^{\perp}_{H^2(X,\Z)}$. Then the 4-tuple $(f^+,f^-,T_X,\phi)$ is an Enriques quadruple of entropy $\log \tau$. This completes the proof of "only if " part of the theorem.

Suppose  $(f^+,f^-,T, \phi)$ is an Enriques quadruple of entropy $\log\tau$. By the three conditions 3)-5), we can apply Lemma \ref{lem:P} to our $T$. Hence, $T_{\R}$ contains an $f^-$-invariant plane $P$ such that $P$ has signature $(2,0)$, $f_{\R}^-|P\in {\rm SO}(P)$, and $P^{\perp}_{T_{\R}}\cap T=0$.
Take an orthonormal basis $u,v$ of $P$. Let $\omega=u+\sqrt{-1}v$. Then $(\omega,\omega)=0$ and $(\omega,\overline{\omega})>0$, and $\omega$ is an eigenvector of $f^-$. Note that $\omega\in (L^-\oplus_{\phi}L^+)\otimes\C$ and $L^-\oplus_{\phi}L^+\cong {\rm II}_{3,19}$. Thus, by surjectivity of Period mapping for complex K3 surfaces, there exist a complex K3 surface $X$, a nonzero holomorphic two form $\omega_X$ on $X$, and an isometry $$F: H^2(X,\Z)\longrightarrow L^-\oplus_{\phi}L^+$$ such that $F (\omega_X)=\omega$. To simplify notations, we identify $H^2(X,\Z)$ with $L^-\oplus_{\phi}L^+$ via $F$. By the choice of $P$,  the sublattice $T$ is the minimal primitive sublattice of $L^-\oplus_{\phi}L^+$ containing $\omega$ after tensoring with $\C$. Thus, $$T_X=T,\,\, {\rm NS}(X)=T^{\perp}_{L^-\oplus_{\phi}L^+},$$ where $T_X$ and ${\rm NS}(X)$ denote the transcendental lattice and N\'eron-Severi lattice of $X$ respectively. 

Choose $h\in L^+$ in the condition 8) of Definition \ref{def:triple}. Then there exists $w\in W(T^{\perp}_{L^-\oplus_{\phi}L^+})$ such that $w(h)$ is an ample class of $X$, where $W(T^{\perp}_{L^-\oplus_{\phi}L^+})$ is the Weyl group of $T^{\perp}_{L^-\oplus_{\phi}L^+}$. Here we use the fact that the ample cone of a projective K3 surface is the fundamental domain of the action on the positive cone by the Weyl group. Let $$\hat{f}:=w\circ (f^-\oplus_{\phi}f^+)\circ w^{-1},\,\,\hat{\sigma}:=w\circ (-{\rm id}_{L^-}\oplus_{\phi}{\rm id}_{L^+})\circ w^{-1}.$$ Then $\hat{f}(w(h))=w((f^-\oplus_{\phi}f^+)(h))$ and $\hat{\sigma}(w(h))=w(h)$ are ample classes of $X$. Note that $\hat{f}\hat{\sigma}=\hat{\sigma}\hat{f}$. Then, by global Torelli Theorem, there exist automorphisms $f,\sigma \in {\rm Aut}(X)$ such that $$\hat{f}=f^*,\,\, \hat{\sigma}=\sigma^*,\,\, f\sigma=\sigma f.$$ Note that $(L^-\oplus_{\phi}L^+)^{\sigma}=L^+$. Thus, ${\rm rk}((L^-\oplus_{\phi}L^+)^{\sigma})=10$ and $l(G((L^-\oplus_{\phi}L^+)^{\sigma}))=10$. Moreover, $L^+$ is a $2$-elementary lattice satisfying $b_{L^+}(x,x)=0\in \Q/\Z$ for any $x\in G(L^+)$ (i.e., the invariant $\delta_{L^+}=0$, see \cite[Definition 4.2.1]{Ni83}).  Then by \cite[Theorem 4.2.2]{Ni83}, the fixed point locus $X^\sigma=\emptyset$. Thus, $\sigma$ is fixed point free, and $f$ descends to an automorphism of the Enriques surface $X/\sigma$ of entropy $\tau$. This completes the proof of the theorem. \end{proof}

\begin{remark}\label{rmk:EnriquesK3}

Any automorphism $\varphi$ of an Enriques surface $S$ admits exactly two liftings, say $\psi_1,\psi_2$, to the covering K3 surface $\tilde{S}$. Moreover, $\psi_1=\psi_2 \sigma$, where $ \sigma$ is the fixed-point free involution of $\tilde{S}$ such that $S=\tilde{S}/\sigma$.  In fact, from the proof of the theorem,  clearly both $(f^+,f^-,T_X,\phi)$  and $(f^+,- f^-,T_X,\phi)$ are Enriques quadruples if one of them is an Enriques quadruple (we will use this observation in Section \ref{sect:nine} (see e.g. the proof of Theorem \ref{thm:rulingout1.42})).

\end{remark}

We conclude this section with the following two lemmas which will be used later.

\begin{lemma}\label{lem:entropyzero}
Let $Y$ be a K3 surface such that $T_Y\cong U\oplus U(2)$. Then any automorphism of $Y$ is of zero entropy.
\end{lemma}

\begin{proof}
Since $Y$ is a 2-elementary K3 surface, $Y$ has a unique automorphism $\theta$ such that $\theta^*|T_Y=-{\rm id}_{T_Y}$ and $\theta^*| {\rm NS}(Y)={\rm id}_{{\rm NS}(Y)}$. By \cite[Theorem 4.2.2]{Ni83}, the fixed locus of $\theta$ is disjoint union of  a smooth elliptic curve $C$ and eight smooth rational curves (see e.g. \cite[Section 4]{BP83} for explicit examples of $Y$ and $\theta$). Let $\varphi $ be any automorphism of $Y$. Since $\theta$ is in the center of ${\rm Aut}(Y)$, $\varphi^*([C])=[C]$. Then by \cite[Theorem 1.4 (1)]{Og07}, $\varphi$ is of zero entropy.  \end{proof}

\begin{lemma}\label{lem:E8(2)positivity}

Let $f\in {\rm O}(L^-)$ be an isometry of finite order such that 1) there exists a $f$-stable primitive sublattice $N\subset L^-$ satisfying $N\cong E_8(2)$ and 2) $T$ is isomorphic to $U\oplus U(2)$, where $T:=N^{\perp}\subset L^-$. Let $g\in {\rm O}(L^+)$ be an isometry with spectral radius $>1$. Then there exists no gluing map $\phi: G(L^-)\longrightarrow G(L^+)$ such that both of the following two statements are true

i) the map $f\oplus g$ extends to $L^-\oplus_{\phi}L^+\cong {\rm II}_{3,19}$, and 

ii) the restriction of $f\oplus_{\phi} g $ to  $N\oplus_{\phi}L^+\subset L^-\oplus_{\phi}L^+$ is positive.
\end{lemma}

\begin{proof}
Suppose otherwise, i.e., there exists a gluing map $\phi: G(L^-)\longrightarrow G(L^+)$ satisfying both i) and ii).

Clearly, we can choose a sufficiently large $n$ such that both $f^n| L^-$ and $\bar{g}^n| G(L^+)$ are identity maps. By ii), the restriction of $f^n\oplus_{\phi} g^n $ to  $N\oplus_{\phi}L^+$ is positive. By Torelli Theorem and surjectivity of Period mapping, there exist an automorphism $F:X\longrightarrow X$ of a K3 surface $X$ and an isometry $\Phi: H^2(X,\Z)\longrightarrow L^-\oplus_{\phi}L^+$ such that 

a) $\Phi\circ F^*=(f^n\oplus_{\phi} g^n) \circ \Phi$, and

b) $\Phi (T_X)=T$, where $T_X$ denotes the transcendental lattice of $X$.

Thus, $F$ is of positive entropy, which contradicts Lemma \ref{lem:entropyzero}. This completes the proof of the lemma. \end{proof}


\section{Realization of $\tau_8$ by an Enriques surface automorphism}\label{sect:eight}
\noindent
In this section, we prove realizability of the Salem number $\tau_8$ in Theorem \ref{thm:main} as the first dynamical degree of an Enriques surface automorphism. Recall that the Salem polynomial of $\tau_8$ is
$$S_8(x):=1 - x^2 - 2x^3 - x^4 + x^6.$$ 

\begin{theorem}\label{thm:minimum}
There exists an automorphism $g: S \longrightarrow S$ of an Enriques surface $S$ such that:

i) The characteristic polynomial of $g^*: H^2(S,\C)\longrightarrow H^2(S,\C)$ is $$(-1 + x)^3 (1 + x)S_8(x)\,\, ;$$

ii) Let $\tilde{S}$ be the universal cover of $S$. Then there is a lifting, say $\tilde{g}: \tilde{S}\longrightarrow \tilde{S}$, of $g$ such that the characteristic polynomial of $\tilde{g}^*:  H^2(\tilde{S},\C)\longrightarrow H^2(\tilde{S},\C)$ is 
$$(-1 + x)^5 (1 + x)^3 (1 + x^2)^2 (1 + x^4)S_8(x)\,\, ;$$ 
and 

iii)The transcendental lattice $T_{\tilde{S}}$ of $\tilde{S}$ is isometric to ${\rm I}_{2,2}(4)$, and the action $\tilde{g}^*|T_{\tilde{S}}$ is of order $8$.

In particular, the entropy of $g$ is $h(g) = \log \tau_8$, i.e., $d_1(g) = \tau_8$. 
\end{theorem}

\noindent
{\bf The Salem factor and the isometry of $E_{10}$}.  Let $(L_0,f_0)$ be the principal $S_8(x)$-lattice (see Section 3). Then $L_0$ is an even lattice of signature $(3,3)$ and $G(L_0)\cong \F_2^2$.  Let $$a=P(f_0+f_0^{-1})\in \Z[f_0+f_0^{-1}],$$ where $P(y)=1+y$.  Note that $1+y$ is a unit of the ring $\Z[y]/(R_8(y))$, where $R_8(y)$ is the trace polynomial of the Salem polynomial of $\tau_8$. Then the twist $L_0(a)$ is an even lattice of signature $(1,5)$ and  $G(L_0(a))\cong \F_2^2$. The order of $\overline{f_0}|G(L_0(a))$ is $2$. The bilinear form $b_{L_0(a)}$ on $G(L_0(a))$ is isomorphic to the bilinear form $-b_{D_4}$ on $G(D_4)$. There exists, up to conjugation, a unique isometry $f_1\in {\rm O}(D_4)$ such that $$\chi_{f_1}(x)=(-1+x)^3(1+x)$$ and the order of $\overline{f_1}|G(D_4)$ is $2$. Then there exists a gluing map $$\phi_1: G(D_4) \longrightarrow G(L_0(a)),$$ such that 
\begin{equation}\label{eq:L^+}
D_4\oplus_{\phi_1}L_0(a)\cong E_{10},
\end{equation}
 and $f_1\oplus f_0$ extends to 
\begin{equation}\label{eq:f^+}
f_1\oplus_{\phi_1} f_0 \in {\rm O}(D_4\oplus_{\phi_1} L_0(a)).
\end{equation}  Let $(L^+,f^+)$ be the pair $((D_4\oplus_{\phi_1}L_0(a))(2), f_1\oplus_{\phi_1}f_0)$. Then the order of the map $\overline{f^+}| G(L^+)$ is $8$, which gives the hint that we need to construct an isometry $f^-\in {\rm O}(L^-)$ satisfying ${\rm Ord}(\overline{f^-}|G(L^-))=8$ (in particular, $8$ divides the order of $f^-$).

\begin{figure}
\xymatrix@+3pc{
&\fbox{        
  $\begin{array}{l}
  {\rm Period \; 2\; lattice\; }D_4 \\  \;\;\;\;\; \;\;\; (0, \; 4) \end{array}$
  }  \ar@{-}[r]^{\displaystyle \F_2^2} &\fbox{        
  $\begin{array}{l}
  {\rm Salem \; factor\;} L_0(a) \\  \;\;\;\;\; \;\;\; (1, \; 5) \end{array}$
  }
  }
\caption{The isometry of $E_{10}$ of spectral radius $\tau_8$}
\end{figure}
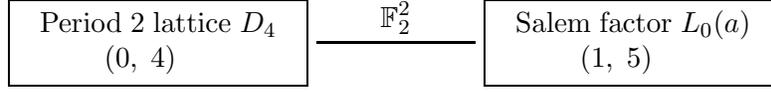

\noindent
{\bf The transcendental factor and the isometry of $U\oplus E_{10}(2)$}. Let $(L_0^\prime,f_0^\prime)$ be the principal $(1+x^4)$-lattice.  Let $(L_2,f_2)$ be the twist $(L_0^\prime(a^\prime),f_0^\prime)$, where $a^\prime=-4-2(f_0^\prime+f_0^{\prime -1})\in \Z[f_0^\prime+f_0^{\prime -1}]$. Then  $L_2\cong {\rm I}_{2,2}(4)$ and $G(L_2)\cong (\Z/4)^4$. There is a pair $(L_3,f_3)$ of an even lattice $L_3$ and an isometry $f_3\in {\rm O}(L_3)$ such that

      i) ${\rm sig}(L_3)=(0,8)$, $G(L_3)\cong (\Z/2)^2\oplus (\Z/4)^4$, $L_3$ has no roots,
      
      ii) $\chi_{f_3}(x)=(-1+x)^2(1+x)^2(1+x^2)^2$,
      
      iii) For $H_1:=\{2x | x\in G(L_2) \}\cong \F_2^4$ and  $H_2:=\{2x | x\in G(L_3) \}\cong \F_2^4$, there exists a gluing map $$\phi_2: H_1 \longrightarrow H_2,$$ such that 
\begin{equation}\label{eq:L^-}
L_2\oplus_{\phi_2}L_3\cong U\oplus E_{10}(2),
\end{equation}
 and $f_2\oplus f_3$ extends to 
\begin{equation}\label{eq:f^-}
f_2\oplus_{\phi_2} f_3 \in {\rm O}(L_2\oplus_{\phi_2} L_3).
\end{equation}  The explicit description of the pair $(L_3,f_3)$ is contained in  a longer version of this paper \cite[Page 22]{OY18} ($L_3$ is in fact a sublattice of $E_8$, and one possible approach to find all rank 8 even negative definite lattices of glue group isomorphic to $(\Z/2)^2\oplus (\Z/4)^4$ is to search for such lattices by considering sublattices in $E_8\oplus E_8(-1) $ generated by eight randomly chosen elements in $E_8\oplus E_8(-1) $, cf. \cite[Theorem 1.1.2]{Ni80}). Let $(L^-,f^-)$ be the pair $(L_2\oplus_{\phi_2}L_3, f_2\oplus_{\phi_2}f_3)$.

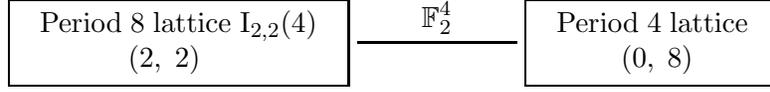
\begin{figure}
\xymatrix@+3pc{
&\fbox{        
  $\begin{array}{l}
  {\rm Period \; 8\; lattice\; }{\rm I}_{2,2}(4) \\  \;\;\;\;\; \;\;\; \;\;\; (2, \; 2) \end{array}$
  }  \ar@{-}[r]^-{\displaystyle \F_2^4} &\fbox{        
  $\begin{array}{l}
  {\rm Period \; 4\; lattice} \\  \;\;\;\;\; \;\;\; (0, \; 8) \end{array}$
  }
  }
\caption{The isometry of $U\oplus E_{10}(2)$}
\end{figure}

\begin{proof}[{\bf Proof of Theorem~\ref{thm:minimum}}]
Let $(L^+,f^+)$, $(L^-,f^-)$ be as constructed above. Then there exists a gluing map $\phi : G(L^-)\longrightarrow G(L^+)$ such that i) $L^-\oplus_{\phi}L^+\cong {\rm II}_{3,19}$ and ii) $f^-\oplus f^+$ extend to $f^-\oplus_{\phi}f^+\in {\rm O}(L^-\oplus_{\phi}L^+)$ (explicit matrix forms of $L^-,L^+,f^-,f^+,\phi$ are included in a longer version of this paper \cite[Pages 23-25]{OY18}). Let $T:=L_2$. Then $T\subset L^-$ is a primitive sublattice, and $T_{L^-}^{\perp}=L_3$. Our goal is to show $(f^+,f^-,T,\phi)$ is an Enriques quadruple of entropy $\log \tau_8$. By construction, all the conditions 1)-7) in Definition \ref{def:triple} are satisfied. Let $L$ denote $T^\perp_{L^-\oplus_{\phi}L^+}$ and let $f=(f^-\oplus_{\phi}f^+)|L.$ It turns out that there exists an $h\in L^+\subset L$ (explicit description of $h$ can be found in \cite[Page 25]{OY18}) such that $h^2=496$ and both $\sR_h$ and $\sS_h$ are empty ($\sR_h$ and $\sS_h$ can be computed using Algorithm \ref{alg:positivity}), which implies  i) none of roots in $L$ is perpendicular to $h$, and ii) $h$ and $f(h)$ are in the same chamber of $L$. Thus, the condition 8) in Definition \ref{def:triple} is satisfied. By Theorem \ref{thm:EnriquesK3}, $\tau_8$ is realized by an automorphism $g: S \longrightarrow S$ of an Enriques surface $S$. Moreover, by the construction,  all the three conditions i)-iii) in Theorem \ref{thm:minimum} are satisfied (see Page 19). \end{proof} 

\begin{remark}
Roughly speaking, the Enriques quadruple $(f^+,f^-,T,\phi)$ in the proof is obtained in the process of trying to rule out $\tau_8$ like ruling out the other 7 Salem numbers in Section \ref{sect:nine} based on Theorem \ref{thm:EnriquesK3} (recall that $\tau_8$ is pseudo-simple, see Remark \ref{rmk:psimple}). However, our statement of Theorem \ref{thm:EnriquesK3} is given in an equivalent form. So, if one obtains a final output, then it is a realization. In this way, we obtained Theorem \ref{thm:minimum}.
\end{remark}


\section{Ruling out smaller Salem numbers}\label{sect:nine}
\noindent

In this section, we prove unrealizability of Salem numbers $\tau_i$ ($1 \le i \le 7$) in Table \ref{tab:8candidates} as the first dynamical degree of an Enriques surface automorphism. Let $S_i(x)$ ($1 \le i \le 7$) be the Salem polynomial of $\tau_i$.

Throughout this section, for a polynomial  $C(x)\in \Z[x]$, we use $\overline{C(x)}\in \F_2[x]$ to denote mod 2 reduction of $C(x)$. We use $L^+$ and $L^-$ to denote $E_{10}(2)$ and $U\oplus E_{10}(2)$ respectively. Let $f^+\in {\rm O}^+(L^+)$, $f^-\in {\rm O}(L^-)$, let $T\subset L^-$ be a primitive sublattice, and let $\phi: G(L^-)\longrightarrow G(L^+)$ be a gluing map. 

We use the following strategy to rule out $\tau_i$ ($1 \le i \le 7$).

\begin{strategy}\label{str:rulingout} Let $\tau$ be a pseudo-simple Salem number (all $\tau_i$ ($1 \le i \le 7$) satisfy this condition, see Remark \ref{rmk:psimple}). Let $S(x)$ denote the Salem polynomial of $\tau$. To show that $\tau$ cannot be realized by any Enriques surface automorphism, it suffices to show that there exists no Enriques quadruple of entropy $\log\tau$ (Theorem \ref{thm:EnriquesK3}). Assuming otherwise that $(f^+,f^-,T,\phi)$ is an Enriques quadruple of entropy $\log\tau$, we derive a contradiction as follows.

1. Find a {\it finite} subset $\sR\subset {\rm O}(E_{10})$ such that i) all $g\in \sR$ satisfy the property that $\chi_g(x)$ is divided by $S(x)$ and ii) any element in ${\rm O}(E_{10})$ satisfying the same property is conjugate in ${\rm O}(E_{10})$ to an  element in $\sR$ (see Remark \ref{rmk:strategy} 3)).  Since $ L^+=E_{10}(2)$ and ${\rm O}(L^+)={\rm O}(E_{10})$, $\sR$ will be viewed as a subset of  ${\rm O}(L^+)$.

2. Let $\sF\subset \F_2[x]$ denote the set consisting of the mod $2$ reduction of the characteristic polynomial of elements in $\sR$. In particular, $\overline{\chi_{f^+}(x)} \in \sF$.

3. Choose some suitable monic polynomial $C_1(x)\in \Z[x]$ such that $\overline{1+x} \nmid \overline{C_1(x)}$,  $\chi_{f^-}(x)=C_1(x)C_2(x)$ (replacing $(f^+,f^-,T,\phi)$ by $(f^+,-f^-,T,\phi)$ if necessary), and  $\overline{C_1(x)}$ and $\overline{C_2(x)}$ are coprime.  Let $L_i:={\rm Ker}(C_i(f^-))\subset L^-$, $f_i:=f^-| L_i$, $r_i:={\rm rk}(L_i)$, $i=1,2$.  By Lemma \ref{lem:L(1/2)}, it turns out that $L_1(1/2)$ is a well-defined even lattice of rank $r_1$ (cf. proof of Theorem \ref{thm:rulingout1.42}). 

4. By using Theorem \ref{thm:L2withroots} (and its variants), we show that the lattice $L_1(1/2)$ is an even negative definite $p$-elementary lattice of rank $r_1$ and determinant $p^l$ for some odd prime $p$ and some $l\ge 0$. Let $\sL_1$  be the (necessarily) finite set consisting of all such $L_1$ up to isomorphism. At this point, glue among the three lattices $L_1,L_2,L^+$  inside $L^-\oplus_{\phi}L^+$ is of the form in the Figure \ref{fig:Lp12}.

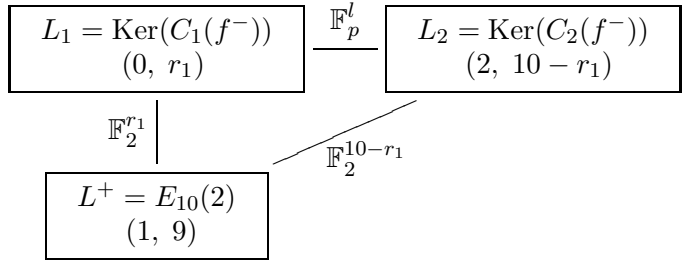
\begin{figure}
\xymatrix{
& & &\fbox{        
  $\begin{array}{l}
 L_1={\rm Ker}(C_1(f^-)) \\  \;\;\;\;\;\;\;\;\;\; (0, \; r_1) \end{array}$
  }  \ar@{-}[r]^{\displaystyle \F_p^l}
&\fbox{        
  $\begin{array}{l}
  L_2={\rm Ker}(C_2(f^-)) \\  \;\;\;\;\; \; (2, \; 10-r_1) \end{array}$
  } \ar@{-}[ld]^{\displaystyle \F_2^{10-r_1}}\\
& & &\fbox{        
  $\begin{array}{l}
  L^+=E_{10}(2) \\  \;\;\;\;\; \; (1, \; 9) \end{array}$
  } \ar@{-}[u]^{\displaystyle \F_2^{r_1}}  }
\caption{ Glue among $L_1,L_2,L^+$ inside $L^-\oplus_{\phi}L^+$ in the strategy of ruling out Salem numbers }\label{fig:Lp12}

\end{figure}

5. For any $L\in \sL_1$, find the conjugacy class $[f]$ of each $f\in {\rm O}(L)$ such that $\chi_{f}(x)=C_1(x)$. Choose one representative $f\in {\rm O}(L)$ for each of such conjugacy classes. We denote by $\sX$ the (necessarily) finite set consisting of all the pairs $(L, f)$ where $L$ runs through $\sL_1$.

6. Show that, for any $g\in \sR\subset {\rm O}(L^+)$, any $(L,f)\in \sX$, and any subgroup $H\subset G(L^+)$, there exists no gluing map $\psi : G(L)_2\longrightarrow H$ such that i) the map $f \oplus g$ extends to $f\oplus_{\psi} g\in {\rm O}(L \oplus_{\psi}L^+)$ and ii) $f\oplus_{\psi} g$ is positive. This contradicts the condition 8) in Definition \ref{def:triple} and therefore $\tau$ is ruled out. \end{strategy}

\begin{remark}\label{rmk:strategy}
1) Unlike the other six cases $\tau_i, i\neq 3 $, our proof of ruling out $\tau_3$ is almost computer free (see proof of Theorem \ref{thm:rulingout1.42}), while it follows Strategy \ref{str:rulingout} partially.

2) For $i=1,2,4,6,7$, we rule out $\tau_i$ exactly following Strategy \ref{str:rulingout} (in each of these five cases, the polynomial $C_1(x)\in \Z[x]$ chosen in the step 3  is irreducible). For $\tau_5$, we also follow Strategy \ref{str:rulingout}, but some differences appear (see the proof of Theorem \ref{thm:rulingout5}). For all $\tau_i$ $(1\le i\le 7, i\neq 3)$, we use  Magma (\cite{BCP}) to find conjugacy classes (and explicit representatives of them) of isometry groups of negative definite lattices in $\sL_1$, and we follow Algorithm \ref{alg:positivity} to check positivity. For computation in these 6 cases, we use a mixture of Mathematica (\cite{Wo}), Magma (\cite{BCP}), PARI/GP (\cite{Th}), and SageMath (\cite{The}).

3) The strategy in \cite[Page 203]{Mc16} may be adapted to find $\sR$ (see Theorem \ref{thm:GluePeriod}). In fact, for each $\tau_i$ ($1\le i\le 7$), we can find a {\it finite} set $\sR$ in this way (the finiteness of $\sR$ is guaranteed by \cite[Corollary 6.3]{Mc16}, see also \cite[Page 298]{BG18}). \end{remark}

Now we start to rule out Salem numbers $\tau_i$ ($1 \le i \le 7$) following Strategy \ref{str:rulingout}. In the rest of this section, we freely use the notation in Strategy \ref{str:rulingout}.

 \begin{theorem}\label{thm:rulingout1.42}
 The Salem number $\tau_3$ cannot be realized by an automorphism of any Enriques surface.
 \end{theorem}

\begin{proof}
 Suppose $(f^+,f^-,T,\phi)$ is an Enriques quadruple of entropy $\log\tau_3$.  For $\tau_3$, it turns out that we can determine $\sF$ without finding $\sR$ explicitly.

\begin{lemma}\label{lem:1.42}
If $g\in \sR$, then $$\overline{\chi_{g}(x)} = \overline{(1 + x)^2(1+x+x^4)(1+x^3+x^4)}.$$
\end{lemma}

\begin{proof}

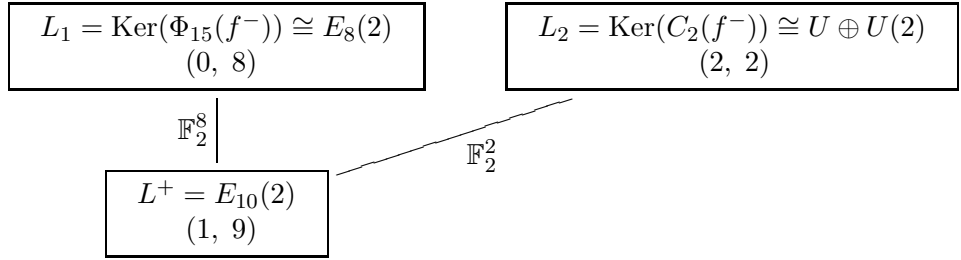
\begin{figure}
\xymatrix{
& &\fbox{        
  $\begin{array}{l}
 L_1={\rm Ker}(\Phi_{15}(f^-))\cong E_8(2) \\  \;\;\;\;\;\;\;\;\;\;\;\;\;\;\; \; \;\;  (0, \; 8) \end{array}$
  }  
&\fbox{        
  $\begin{array}{l}
  L_2={\rm Ker}(C_2(f^-))\cong U\oplus U(2) \\  \;\;\;\;\; \; \;\;\;\;\; \;\;\;\;\;\; \; \;\;   (2, \; 2) \end{array}$
  } \ar@{-}[ld]^{\displaystyle \F_2^{2}}\\
& &\fbox{        
  $\begin{array}{l}
  L^+=E_{10}(2) \\  \;\;\;\;\; \; (1, \; 9) \end{array}$
  } \ar@{-}[u]^{\displaystyle \F_2^{8}}  }
\caption{ Ruling out $\tau_3$ }\label{fig:tau3}
\end{figure}

Since $g\in\sR$, $\chi_g(x)=S_3(x)Q(x)$, where $Q(x)$ is a product of cyclotomic polynomials, and ${\rm deg}(Q(x))=2$. There are exactly five cyclotomic polynomials $\Phi_k(x)$ ($k=1,2,3,4,6$) with degree less than or equal to 2. By computation, $|{\rm res}(\Phi_k(x),S_3(x))|=1$ for $k=3, 6$. Then $Q(x)\neq \Phi_3(x),\Phi_6(x)$ (otherwise, by \cite[Theorem 4.3]{Mc16}, the Salem factor ${\rm Ker}(S_3(g))\subset E_{10}$ must be an even unimodular lattice of signature $(1,7)$, a contradiction). This implies Lemma \ref{lem:1.42}. \end{proof}

We choose $C_1(x)$ to be $\Phi_{15}(x)$. By Lemma \ref{lem:1.42} and 2) of Definition \ref{def:triple}, $$\overline{\chi_{f^-}(x)}=\overline{(1+x)^2\chi_{f^+}(x)}=\overline{(1+x)^4C_1(x)}.$$ Since $f^-$ is of finite order by 3) of Definition \ref{def:triple}, $\chi_{f^-}(x)$ is a product of cyclotomic polynomials of degree $\le12$. Then by factorizations of such cyclotomic polynomials in $\F_2[x]$ (see Table \ref{tab:cycpoly}), either $\Phi_{30}(x)$ or $C_1(x)=\Phi_{15}(x)$ divides $\chi_{f^-}(x)$. Replacing $f^-$ by $-f^-$ if necessary (see Remark \ref{rmk:EnriquesK3}), we may and will assume that $C_1(x)$ divides $\chi_{f^-}(x)$.  Since $\overline{C_1(x)}$ and $\overline{C_2(x)}$ are coprime in $\F_2[x]$ (equivalently, $2\nmid {\rm res}(C_1(x),C_2(x))$), by Lemma \ref{lem:sylowp}, we have $$\chi_{\overline{f^-}}(x)=\chi_{\overline{f_1}|G(L_1)_2}(x)\chi_{\overline{f_2}|G(L_2)_2}(x),$$ $G(L_i)_2\cong \F_2^{k_i}$, where $k_1,k_2\ge 0$ and $k_1+k_2=10$. Since $\phi: G(L^-)\longrightarrow G(L^+)$ is a gluing map, by 7) of Definition \ref{def:triple}, $\chi_{\overline{f^-}}(x)=\chi_{\overline{f^+}}(x)=\overline{\chi_{f^+}(x)}$. Thus, $$\chi_{\overline{f_1}|G(L_1)_2}(x)\chi_{\overline{f_2}|G(L_2)_2}(x) = \overline{(1 + x)^2C_1(x)}.$$ It follows that $$\chi_{\overline{f_1}|G(L_1)_2}(x)=\overline{C_1(x)},\; \chi_{\overline{f_2}|G(L_2)_2}(x)=\overline{(1 + x)^2}.$$ Thus, $k_1=8$ and $k_2=2$. Then, by Lemma \ref{lem:L(1/2)},  $L_1(1/2)$ is a well-defined even lattice. 

By Table \ref{tab:cycpoly}, $C_2(x)$ is a product of polynomials in $\{\Phi_1(x),\Phi_2(x),\Phi_4(x),\Phi_8(x)\}$, which implies $|{\rm res}(C_1(x),C_2(x))|=1.$ Thus, $L^-=L_1\oplus L_2$ by \cite[Proposition 4.2]{Mc16}.  Then $G(L_i)=G(L_i)_2=\F_2^{k_i}$. Thus, $L_1(1/2)$ is unimodular. By 5) of Definition \ref{def:triple}, either $T\subset L_1$ or $T\subset L_2$, and by 4) of Definition \ref{def:triple},  ${\rm sig}(L_1(1/2))={\rm sig}(L_1)$ is either $(2,6)$ or $(0,8)$. Then by classification of even unimodular lattice of rank $8$ (see e.g. \cite[Chapter V]{Se73}), we have ${\rm sig}(L_1(1/2))=(0,8)$ and $L_1(1/2)\cong E_{8}$. Hence $L_1\cong E_{8}(2)$, $T\subset L_2$, and ${\rm sig}(L_2)=(2,2)$. By using \cite[Theorem 4.3.1]{Ni83},  we deduce that $L_2\cong U\oplus U(2)$. See Figure \ref{fig:tau3}.

Let $R\subset L^-\oplus_{\phi} L^+$ be the smallest primitive sublattice containing both $L_1$ and $L^+$. By 8) of Definition \ref{def:triple},  the restriction of $f^-\oplus_{\phi} f^+$ to  $T^{\perp}_{L^-\oplus_{\phi} L^+}$ is positive. Since $R\subset T^{\perp}_{L^-\oplus_{\phi} L^+}$, by Corollary \ref{cor:subpositive}, $(f^-\oplus_{\phi} f^+)| R$ is positive. This contradicts Lemma \ref{lem:E8(2)positivity}. \end{proof}

\begin{theorem}\label{thm:rulingout12467}

None of the Salem numbers $\tau_i$ ($i=1, 2,4,6,7$) can be realized by an automorphism of an Enriques surface.
\end{theorem}

\begin{proof}
 First, we rule out $\tau_1$ by following Strategy \ref{str:rulingout} and by using computer algebra (see tau1.txt in enriques.zip available at \cite{Yu} for computational data). We set $K:=\Q[x]/(S_1(x))$ and $k:=\Q[x+x^{-1}]\subset K$ (See section 3). Let $(L_0,f_0)$ denote the principal $S_1(x)$-lattice.  Let $\sU\subset \sO_k^{\times}$ denote a set of representatives for the units modulo squares. There are exactly four units $u_i\in \sU$, $i=1,2,3,4$ such that the twists $L_0(u_i)$ are isomorphic to $E_{10}$. Since $\tau_1$ is simple, by Theorem \ref{thm:twist}, up to conjugation, there are at most four isometries, say $g_i\in {\rm O}(L^+)$, $i=1,2,3,4$ with characteristic polynomial $S_1(x)$.  We set $\sR:=\{g_1,g_2,g_3,g_4\}$. Note that $\sF=\{\overline{(1+x+x^2)^2\Phi_7(x)}\}$.

We choose $C_1(x)$ to be $\Phi_7(x)$. By the same argument as in the proof of Theorem \ref{thm:rulingout1.42},   we deduce that $G(L_1)_2\cong \F_2^6$, $G(L_2)_2\cong \F_2^4$, and $L_1(1/2)$ is a well-defined even lattice.

Then by Theorem \ref{thm:L2withroots} ($n_1=7$), it follows that $G(L_1)=\F_2^6\oplus \F_7,\,\,G(L_2)=\F_2^{4}\oplus \F_7$. Moreover, if ${\rm sig}(L_1)=(2,4)$, then $L_2$ has roots. Then, by 6) of Definition \ref{def:triple}, ${\rm sig}(L_1)=(0,6)$. Thus, $L_1(1/2)$ is an even negative definite lattice of determinant $7$ and rank $6$. By classification, $L_1(1/2)\cong A_6$ (see \cite[Table 1]{CS88} ). Thus, we obtain $\sL_1=\{A_6(2)\}$.  Up to conjugation in ${\rm O}(A_6(2))$, there is a unique isometry $f\in {\rm O}(A_6(2))$ with characteristic polynomial $C_1(x)$. Thus, $\sX=\{(A_6(2), f)\}$.

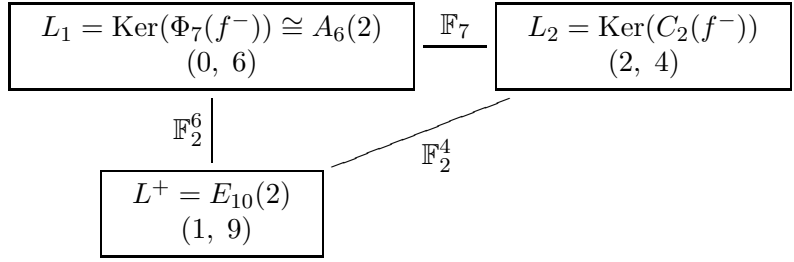
\begin{figure}
\xymatrix{
& & &\fbox{        
  $\begin{array}{l}
 L_1={\rm Ker}(\Phi_7(f^-))\cong A_6(2) \\  \;\;\;\;\;\;\;\;\;\;\;\;\;\;\; \;\;\;  (0, \; 6) \end{array}$
  }  \ar@{-}[r]^{\;\;\;\;\;\;\;\;\;\displaystyle \F_7}
&\fbox{        
  $\begin{array}{l}
  L_2={\rm Ker}(C_2(f^-)) \\ \;\;\;\;\; \;  \;\;\;\;(2, \; 4) \end{array}$
  } \ar@{-}[ld]^{\displaystyle \F_2^{4}}\\
& & &\fbox{        
  $\begin{array}{l}
  L^+=E_{10}(2) \\  \;\;\;\;\; \; (1, \; 9) \end{array}$
  } \ar@{-}[u]^{\displaystyle \F_2^{6}}  }
\caption{ Ruling out $\tau_1$ }\label{fig:tau1}
\end{figure}

Using computer algebra (see Remark \ref{rmk:strategy} 2)), for each $g\in \sR$, there are exactly seven pairs $(H, \psi)$, where $H\subset G(L^+)$ is a subgroup of order $2^6$ and $\psi: G(A_6(2))_2\longrightarrow H$ is a gluing map, such that  the map $f\oplus g$ extends to $f\oplus_{\psi} g\in {\rm O}(A_6(2)\oplus_{\psi}L^+)$. Thus, totally, there are 28 candidates $f\oplus_{\psi} g$, but it turns out that none of them is positive by positivity test (see Algorithm \ref{alg:positivity}), a contradiction to  $(f^-\oplus_{\phi} f^+)| T_{L^-\oplus_{\phi} L^+}^\perp$ being positive. This completes the proof of ruling out $\tau_1$.

Similar to ruling out $\tau_1$, we follow Strategy \ref{str:rulingout} to rule out the four Salem numbers $\tau_{i}$ ($i=2,4,6,7$). In the step 3 of Strategy \ref{str:rulingout},  we choose $\Phi_5(x), \Phi_9(x), \Phi_9(x), \Phi_7(x)$ for $C_1(x)$ respectively for $\tau_{i}$ ($i=2,4,6,7$). It turns out that $\sL_1$ for $\tau_{i}$ ($i=2,4,6,7$) are $\{ A_4(2)\}$,  $\{ E_6(2)\}$, $\{ E_6(2)\}$, $\{ A_6(2)\}$ respectively. We omit the details of the proof for these $4$ cases (see \cite[Section 9]{OY18} for more details, and see the four text files tau2.txt, tau4.txt, tau6.txt, tau7.txt in enriques.zip available at \cite{Yu} for computational data). \end{proof}

Now the remaining case is $\tau_5$. This case is the most complicated case.  The main reason is that the mod $2$ reduction of $S_5(x),$  $$\overline{S_5(x)}=\overline{(1+x+x^2)^3},$$ has only one irreducible factor of small degree, so that there are more candidates for cyclotomic factors to be checked.

\begin{theorem}\label{thm:rulingout5}

The Salem number $\tau_5$ cannot be realized by an automorphism of any Enriques surface.
\end{theorem}

\begin{proof}
 Again we follow Strategy \ref{str:rulingout}, although the way to use it will be slightly different from the previous cases. We only sketch the proof below (see \cite[Section 9]{OY18} for the details, and see tau5.txt in enriques.zip available at \cite{Yu} for computational data) and point out the main differences which appear. 

The feasible primes for $\tau_5$ are $2$ and $5$, and $|{\rm det}(L_0)|=5$ where $(L_0,f_0)$ is the principal $S_5(x)$-lattice. By Theorem \ref{thm:GluePeriod},  a necessary condition for a twist being a Salem factor of an isometry of $E_{10}$ is $|{\rm det}(L_0(a))|\in\{5, 20\}$. Then we obtain $\sR$ via gluing suitable twists $(L_0(a),f_0)$ to  isometries of rank $4$ even negative definite lattices of determinant $5$ or $20$ (cf. \cite{CS88}, \cite{Nip91}).  It turns out $$\sF=\{\overline{(1+x)^4\Phi_3(x)^3}, \overline{\Phi_3(x)^3\Phi_5(x)}, \overline{(1+x)^2\Phi_3(x)^4}\}.$$ In order to apply the four steps 3-6 of Strategy \ref{str:rulingout} effectively, we need to divide $\sR$ into two subsets (which is different from the previous cases). Let $$\sR_3:=\{g\in \sR | \overline{\chi_g(x)}= \overline{(1+x)^4\Phi_3(x)^3}\;\text{or}\;\,  \overline{\Phi_5(x)\Phi_3(x)^3}\},$$ $$\sR_4:=\{g\in \sR |\overline{ \chi_g(x)}=\overline{ (1+x)^2\Phi_3(x)^4} \}.$$ We consider two cases separately:  Case $f^+\in \sR_3$ and Case $f^+\in \sR_4$. 

Case $f^+\in \sR_3$ :  We choose $C_1(x)$ to be any product of cyclotomic polynomials in $\{\Phi_3(x), \Phi_6(x),\Phi_{12}(x)\}$ such that $\overline{C_1(x)}=\overline{(1+x+x^2)^3}$ (here $C_1(x)$ is not necessarily an irreducible polynomial). It turns out that $\sL_1=\{E_6(2), A_2(2)^{\oplus 3}\}$ (here we need a variant of Theorem \ref{thm:L2withroots} to conclude) and  no gluing map $\psi$ with required properties in the step 6 of Strategy \ref{str:rulingout} exists. Thus, the case $f^+\in \sR_3$ is impossible.

\begin{figure}
\xymatrix{
& & &\fbox{        
  $\begin{array}{l}
 L_1\in\{E_6(2), A_2(2)^{\oplus 3}\} \\  \;\;\;\;\;\;\;\;\;\;\;\;\;\;\;  (0, \; 6) \end{array}$
  }  \ar@{-}[r]^{\;\;\;\;\;\;\;\displaystyle \F_3^l}
&\fbox{        
  $\begin{array}{l}
  L_2={\rm Ker}(C_2(f^-)) \\ \;\;\;\;\; \;  \;\;\;\;(2, \; 4) \end{array}$
  } \ar@{-}[ld]^{\displaystyle \F_2^{4}}\\
& & &\fbox{        
  $\begin{array}{l}
  L^+=E_{10}(2) \\  \;\;\;\;\; \;\; (1, \; 9) \end{array}$
  } \ar@{-}[u]^{\displaystyle \F_2^{6}}  }
\caption{ Ruling out $\tau_5$: case $f^+\in \sR_3$, $\overline{C_1(x)}=\overline{\Phi_3(x)^3}$ }\label{fig:tau5R3}
\end{figure}
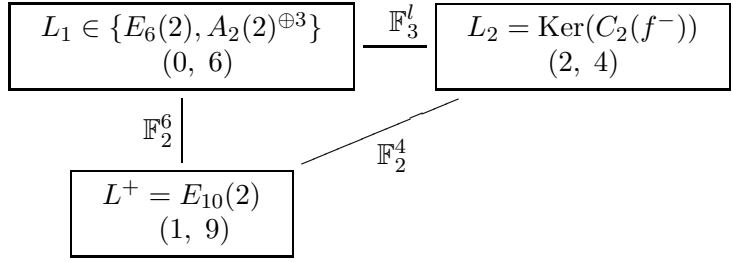

Case $f^+\in\sR_4$ : In this case,  we choose $C_1(x)$ to be any product of cyclotomic polynomials in $\{\Phi_3(x), \Phi_6(x),\Phi_{12}(x),\Phi_{24}(x)\}$ such that $\overline{C_1(x)}=\overline{(1+x+x^2)^4}$. As in the previous case, it turns out that  $$\sL_1=\{E_8(2), E_6(2)\oplus A_2(2), A_2(2)^{\oplus 4}, M^\prime(2)\},$$ where $M^\prime(2)$ is the unique (up to isomorphism) rank $8$ even negative definite lattice of the same discriminant form as $A_2(2)^{\oplus 4}$ but not isometric to $A_2(2)^{\oplus 4}$.

 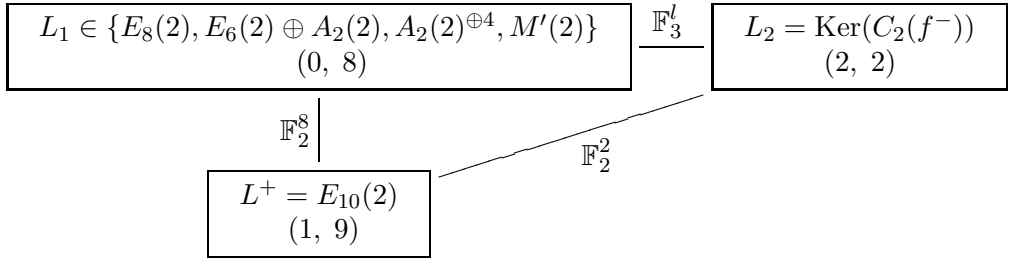
\begin{figure}
\xymatrix{
&\fbox{        
  $\begin{array}{l}
 L_1\in\{E_8(2), E_6(2)\oplus A_2(2), A_2(2)^{\oplus 4}, M^\prime(2)\} \\  \;\;\;\;\;\;\;\;\;\;\; \;\;\;\;\;\;\;\;\;\;\;\;\;\;\;\;\;\; \;\;\;  (0, \; 8) \end{array}$
  }  \ar@{-}[r]^{\;\;\;\;\;\;\;\;\;\;\;\;\;\;\;\;\;\;\;\;\;\;\;\;\;\displaystyle \F_3^l}
&\fbox{        
  $\begin{array}{l}
  L_2={\rm Ker}(C_2(f^-)) \\ \;\;\;\;\; \;  \;\;\;\;(2, \; 2) \end{array}$
  } \ar@{-}[ld]^{\displaystyle \F_2^{2}}\\
&\fbox{        
  $\begin{array}{l}
  L^+=E_{10}(2) \\  \;\;\;\;\; \; (1, \; 9) \end{array}$
  } \ar@{-}[u]^{ \;\;\;\;\;\;\;\;\;\;\;\;\;\;\; \;\;\;\displaystyle \F_2^{8}}  }
\caption{ Ruling out $\tau_5$: case $f^+\in \sR_4$,  $\overline{C_1(x)}=\overline{\Phi_3(x)^4}$ }\label{fig:tau5R4}
\end{figure}

For $L_1\in \sL_1\setminus \{A_2(2)^{\oplus 4}\}$, it turns out that no gluing map $\psi$ with required properties in the step 6 of Strategy \ref{str:rulingout} exists. However, if $L_1\cong A_2(2)^{\oplus 4}$, then there exist $\psi$ satisfying both i) and ii) in the step 6 of Strategy \ref{str:rulingout} (this is a new difference). On the other hand, if $L_1\cong A_2(2)^{\oplus 4}$, then by using Lemma \ref{lem:L(1/2)}, we deduce that $L_2(1/3)\cong U\oplus U(2)$ and $L_2\cong U(3)\oplus U(6)$. However, it turns out that there exists no isometry of $U(3)\oplus U(6)$ which is glued to some $f\oplus_{\psi} g\in {\rm O}(L_1 \oplus_{\psi}L^+)$ to give an Enriques quadruple. Therefore, the case $f^+\in \sR_4$ is impossible. This completes the proof.\end{proof}

\appendix

\section{Tables}

    \begin{table}[htp]
\caption{The eight candidate small Salem numbers $\tau_i$}\label{tab:8candidates}
\begin{center}
{\footnotesize

\begin{tabular}{|c| p{1.5cm}|p{5cm}|p{5.5cm}|}
\hline 
  & value &$S_i(x):=$ minimal polynomial of $\tau_i$ &factorization of $S_i(x)$ in $\F_2[x]$\\
\hline 
 $\tau_1$ & $1.35098...$&$1-x-x^4+x^5-x^6-x^9+x^{10}$&$(1+x+x^2)^2(1+x+x^3)(1+x^2+x^3)$\\
\hline 
 $\tau_2$& $1.40126...$ &$1-x^2-x^3-x^4+x^6$& $(1+x+x^2)(1+x+x^2+x^3+x^4)$\\
\hline 
 $\tau_3$ & $1.42500... $ &$1-x-x^3+x^4-x^5-x^7+x^8$&$(1+x+x^4)(1+x^3+x^4)$\\
\hline 
 $\tau_4$ & $1.45798...$ &$1-x^2-x^3-x^5-x^6+x^8$&$(1+x)^2(1+x^3+x^6)$\\
\hline 
 $\tau_5$ & $1.50613...$  &$1-x-x^3-x^5+x^6$&$(1+x+x^2)^3$\\
\hline 
 $\tau_6$ & $1.53292...$  &$1-x-x^2+x^5-x^8-x^9+x^{10}$&$(1+x+x^2+x^3+x^4)(1+x^3+x^6)$\\
\hline 
$\tau_7$& $1.55603...$  &$1-x-x^2+x^3-x^4-x^5+x^6$&$(1+x+x^3)(1+x^2+x^3)$\\
\hline 
$\tau_8$& $1.58234...$  &$1-x^2-2x^3-x^4+x^6$&$(1+x)^6$\\
\hline 

\end{tabular}
}
\end{center}

\end{table}

    \begin{table}[htp]
\caption{Cyclotomic polynomials of degree $\le 12$ and their reductions mod $2$}\label{tab:cycpoly}
\begin{center}
{\footnotesize

\begin{tabular}{|c| p{6cm}|p{6cm}|}
\hline 
 $i$ &$ \Phi_i(x)$  &factorization of $\Phi_i(x)$ in $\F_2[x]$\\
\hline 
 $1$ & $-1+x$&$1+x$\\
\hline 
 $2$& $1+x$ &$1+x$\\
\hline 
 $3$ & $1+x+x^2 $ &$1+x+x^2$\\
\hline 
 $4$ & $1+x^2$ &$(1+x)^2$\\
\hline 
 $5$ & $1+x+x^2+x^3+x^4$  &$1+x+x^2+x^3+x^4$\\
\hline 
 $6$ & $1-x+x^2$  &$1+x+x^2$\\
\hline 
$7$& $1+x+x^2+x^3+x^4+x^5+x^6$  &$(1+x+x^3) (1+x^2+x^3)$\\
\hline 
 $8$& $1+x^4$ &$(1+x)^4$\\
\hline 
$9$& $1+x^3+x^6$ &$1+x^3+x^6$\\
\hline 
$10$& $1-x+x^2-x^3+x^4$  &$1+x+x^2+x^3+x^4$\\
\hline 
 $11$ & $1+x+x^2+x^3+x^4+x^5+x^6+x^7+x^8+x^9+x^{10}$&$1+x+x^2+x^3+x^4+x^5+x^6+x^7+x^8+x^9+x^{10}$\\
\hline 
 $12$& $1-x^2+x^4$ &$(1+x+x^2)^2$\\
\hline 
 $13$ & $1+x+x^2+x^3+x^4+x^5+x^6+x^7+x^8+x^9+x^{10}+x^{11}+x^{12} $ &$1+x+x^2+x^3+x^4+x^5+x^6+x^7+x^8+x^9+x^{10}+x^{11}+x^{12}$\\
\hline 
 $14$ & $1-x+x^2-x^3+x^4-x^5+x^6$ &$(1+x+x^3) (1+x^2+x^3)$\\
\hline 
 $15$ & $1-x+x^3-x^4+x^5-x^7+x^8$  &$(1+x+x^4) (1+x^3+x^4)$\\
\hline 
 $16$ & $1+x^8$  &$(1+x)^8$\\
\hline 
$18$& $1-x^3+x^6$  &$1+x^3+x^6$\\
\hline 
 $20$& $1-x^2+x^4-x^6+x^8$ &$(1+x+x^2+x^3+x^4)^2$\\
\hline 
$21$& $1-x+x^3-x^4+x^6-x^8+x^9-x^{11}+x^{12}$ &$(1+x+x^2+x^4+x^6) (1+x^2+x^4+x^5+x^6)$\\
\hline 
$22$& $1-x+x^2-x^3+x^4-x^5+x^6-x^7+x^8-x^9+x^{10}$ &$1+x+x^2+x^3+x^4+x^5+x^6+x^7+x^8+x^9+x^{10}$\\
\hline 
 $24$ & $1-x^4+x^8$ &$(1+x+x^2)^4$\\
\hline 
 $26$ & $1-x+x^2-x^3+x^4-x^5+x^6-x^7+x^8-x^9+x^{10}-x^{11}+x^{12}$ &$1+x+x^2+x^3+x^4+x^5+x^6+x^7+x^8+x^9+x^{10}+x^{11}+x^{12}$\\
\hline 
 $28$ & $1-x^2+x^4-x^6+x^8-x^{10}+x^{12}$  &$(1+x+x^3)^2 (1+x^2+x^3)^2$\\
\hline 
 $30$ & $1+x-x^3-x^4-x^5+x^7+x^8$  &$(1+x+x^4) (1+x^3+x^4)$\\
\hline 
$36$& $1-x^6+x^{12}$  &$(1+x^3+x^6)^2$\\
\hline 
 $42$& $1+x-x^3-x^4+x^6-x^8-x^9+x^{11}+x^{12}$ &$(1+x+x^2+x^4+x^6) (1+x^2+x^4+x^5+x^6)$\\
\hline

\end{tabular}
}
\end{center}

\end{table}

    \begin{table}[htp]
\caption{The  minimum Salem number $\lambda_{2d}$ in degree $2d$}\label{tab:smallest}
\begin{center}
{\footnotesize

\begin{tabular}{|c| p{1.5cm}|p{5cm}|p{5.5cm}|}
\hline 
  & value & minimal polynomial of $\lambda_{2d}$\\
\hline 
 $\lambda_{10}$ & $1.17628...$&$1+x-x^3-x^4-x^5-x^6-x^7+x^9+x^{10}$\\
\hline 
 $\lambda_8$& $1.28063...$ &$1-x^3-x^4-x^5+x^8$\\
\hline 
 $\lambda_6$ & $1.40126... $ &$1-x^2-x^3-x^4+x^6$\\
\hline 
 $\lambda_4$ & $1.72208...$ &$1-x-x^2-x^3+x^4$\\
\hline 
 $\lambda_2$ & $2.61803...$  &$1-3x+x^2$\\
\hline 

\end{tabular}
}
\end{center}

\end{table}

\section{Proof of Lemma \ref{lem:P}}

In this Appendix, we prove Lemma \ref{lem:P} for the sake of completeness. 

By the assumption on $f$, the characteristic polynomial $\chi_{f}(x)=\Phi_k^n(x)$, where $k={\rm Ord}(f)$. The lemma is true when $k=1,2$ (see \cite[Lemma 2.13]{Og02}). In fact, we may choose a $\Q$-basis $e_1,e_2,...,e_{2+r}$ of $T_{\Q}$ such that $(e_i,e_i)>0$ for $i=1,2$, $(e_i,e_i)<0$ for $i=3,...,2+r$, and $(e_i,e_j)=0$ for $i\neq j$. Let $t\in \R$ be a transcendental number such that $t>1$. Let $N>r$ be a sufficiently large integer such that $$v_1:=t^Ne_1+te_3+...+t^re_{2+r}$$ satisfying $v_1^2>0$. Let $P\subset T_{\R}$ be the plane generated by $v_1, e_2$. Then $P$ is of signature $(2,0)$ and $P^{\perp}_{T_{\R}}\cap T=0$.

From now on, we assume $k>2$. 

Suppose $n=1$. Then $\chi_{f}(x)=\Phi_k(x)$, and $\Z[x]/(\Phi_k(x))$ is a PID (here we use the fact ${\rm deg}(\Phi_k(x))=2+r\le 12$).  $T$ is a free $\Z[x]/(\Phi_k(x))$-module of rank 1. Note that ${\rm rk}(T)=2+r=2m$ for some positive integer $m$. Thus, the roots of $\Phi_k(x)$ are of the form $\xi_1,\bar{\xi_1},..., \xi_m,\bar{\xi_m}$. Then we have the following decomposition 
\begin{equation}\label{equ:dec}
T_{\C}=V(\xi_1)\oplus V(\bar{\xi_1})\oplus\cdots \oplus V(\xi_m)\oplus V(\bar{\xi_m})  =(V(\xi_1)\oplus V(\bar{\xi_1}))\perp\cdots \perp (V(\xi_m)\oplus V(\bar{\xi_m})),
\end{equation}
 where $V(\xi_i)$ (resp. $V(\bar{\xi_i})$ )denotes the one-dimensional eigenspace of $f| T$ with respect to $\xi_i$ (resp. $\bar{\xi_i}$). For any $i$, we choose a nonzero $v_i\in V(\xi_i)$. Then $\overline{v_i}\in V(\overline{\xi_i})$. We write $v_i=x_i+\sqrt{-1} y_i$, where $x_i,y_i\in T_{\R}$. By $$(v_i,v_i)=(f(v_i),f(v_i))=(\xi_i v_i,\xi_i v_i)=\xi_i^2(v_i,v_i)$$ and $\xi_i^2\neq 1$, we have $(v_i,v_i)=0$. Similarly, $(\overline{v_i},\overline{v_i})=0$. Thus, $(x_i,x_i)=(y_i,y_i)$, and $(x_i,y_i)=0$. Let $a_i:=(x_i,x_i)$. Then $(v_i,\overline{v_i})=2a_i$. Thus, the intersection matrix on $V(\xi_i)\oplus V(\bar{\xi_i})$ with respect to the basis $x_i,y_i$ is $\begin{pmatrix} 
    a_i&0 \\ 
    0&a_i\\ 
     \end{pmatrix}$. Since $T$ is of signature $(2,r)$, it follows that there exists  a unique $i$, say $1$, such that $a_i>0$. Then we choose $P:=\R\langle x_1, y_1\rangle$. Note that $P\otimes \C=V(\xi_1)\oplus V(\bar{\xi_1})$. By $f(v_1)=\xi_1 v_1, f(\overline{v_1})=\overline{\xi_1}\overline{v_1}, \xi_1\overline{\xi_1}=1$, we have $f(P)=P$ and $f_{\R}|P\in {\rm SO}(P)$. For any $x\in P^{\perp}_{T_{\R}}\cap T$, we have $(v_1,x)=0$ and $(\overline{v_1},x)=0$. Since the Galois group $Gal(\Q(\xi_1)/\Q)$ acts on $\{\xi_1,\bar{\xi_1},..., \xi_m,\bar{\xi_m}\}$ transitively, it follows that $(v_i,x)=(\overline{v_i},x)=0$ for any $i$. Thus, $x=0$ since $T$ is a non-degenerate lattice.

Suppose $n>1$. Let $s={\rm deg}(\Phi_k(x))$. We choose any $v\in T$ such that $v^2>0$. Let $L_1\subset T$ be the sublattice generated by $v,f(v),...,f^{s-1}(v)$. By ${\rm Ord}(f|T)=k$ and the minimal polynomial of $f$ is $\Phi_k(x)$, it follows that $f(L_1)=L_1$ and $\chi_{f|L_1}(x)=\Phi_{k}(x)$. Then by considering decomposition as in (\ref{equ:dec}), we deduce that $L_1$ is of signature $(2,s-2)$. Then $(L_1)_{T}^\perp$ is negative definite. We choose any nonzero $v_2\in (L_1)_{T}^\perp$, and let $L_2\subset T$ be the sublattice generated by  $v_2,f(v_2),...,f^{s-1}(v_2)$. Then $L_2$ is a negative definite lattice such that $f(L_2)=L_2$ and $\chi_{f|L_2}(x)=\Phi_k(x)$. Repeating this process, we obtain sublattices $L_i$, $i=1,2,...,n$ such that the following conditions 1) - 4) are satisfied:

1) $L_i\perp L_j$ for $i\neq j$, 

2) $f(L_i)=L_i$ for any $i$, 

3) ${\rm sig}(L_1)=(2,s-2)$ and ${\rm sig}(L_i)=(0,s)$ for $i>1$ and, 

4) $\chi_{f|L_i}(x)=\Phi_{k}(x)$ for any $i$.

Then $L_1\oplus\cdots\oplus L_n$ is a sublattice of $T$ of finite index. Then there exists a root, say $\xi$, of $\Phi_k(x)$ such that 1) $f(w_1)=\xi w_1$ for some nonzero $w_1\in (L_1)_{\Q(\xi)}$ and 2) $\R\langle x_1,y_1\rangle$ is of signature $(2,0)$, where $w_1=x_1+\sqrt{-1}y_1$.  For any $i\ge 2$, we choose a nonzero $w_i=x_i+\sqrt{-1}y_i \in (L_i)_{\Q(\xi)}$ such that $f(w_i)=\xi w_i$. Let $t\in \R$ be any transcendental number such that $t>1$. For a sufficiently large integer $N>n$, the plane $$P:=\R\langle t^Nx_1+tx_2+\cdots +t^{n-1}x_n,  t^Ny_1+ty_2+\cdots +t^{n-1}y_n\rangle \subset T_{\R}$$ is of signature $(2,0)$. Then, one can verify that $f(P)=P$, $f_{\R}|P\in {\rm SO}(P)$, and $P^{\perp}_{T_\R}\cap T=0$. This completes the proof of the lemma.

\end{document}